\documentclass[11pt]{article}
\usepackage{graphicx}
\usepackage{amsthm, amsmath, amssymb, tikz, bm}
\usepackage{mathtools}
\usepackage{xifthen}
\usepackage{comment}

\usepackage[margin=1in]{geometry} 

\usepackage[shortlabels]{enumitem}
\usepackage{todonotes}
\usepackage[colorlinks=true,
linkcolor=blue,citecolor=blue,
urlcolor=blue]{hyperref}
\usepackage[normalem]{ulem}

\usetikzlibrary{calc,shapes, backgrounds}

\newtheorem{theorem}{Theorem}[section]
\newtheorem{corollary}[theorem]{Corollary}

\newtheorem{lemma}[theorem]{Lemma}
\newtheorem{claim}{Claim}

\newtheorem{conjecture}{Conjecture}

\newcommand{\cH}{\mathcal{H}}

\makeatletter
               {\list{}{\leftmargin=0pt 
                        \labelwidth\z@ \itemindent-\leftmargin
                        }}%
               {\endlist}
\makeatother

\title{On the 3-colorability of triangle-free and fork-free graphs}

\author{
Joshua Schroeder \thanks{Georgia Institute of Technology, Atlanta, GA 30332,
({\tt joschroed@gmail.com}).}
\and
Zhiyu Wang \thanks{Louisiana State University, Baton Rouge, LA 70803,
({\tt zhiyuw@lsu.edu}). This author was supported in part by LA Board of Regents grant LEQSF(2024-27)-RD-A-16.}
\and
Xingxing Yu \thanks{Georgia Institute of Technology, Atlanta, GA 30332,
({\tt yu@math.gatech.edu}). This author was partially supported by NSF Grant DMS-1954134 and NSF Grant DMS-2348702.}
}

\begin{document}

\maketitle

\begin{abstract}
A graph $G$ is said to satisfy the Vizing bound if $\chi(G)\leq \omega(G)+1$,  where $\chi(G)$ and $\omega(G)$ denote  the  chromatic  number  and  clique  number of $G$, respectively. It was conjectured by Randerath in 1998 that if $G$ is a triangle-free and fork-free graph, where the fork (also known as trident) is obtained from $K_{1,4}$ by subdividing two edges, then $G$ satisfies the Vizing bound.  In this paper, we confirm this conjecture. 
\end{abstract}

\section{Introduction}

The \emph{chromatic number} $\chi(G)$ of a graph $G$ is the least number of colors needed to properly color the vertices of $G$. 
Given a family of graphs $\cH$, a graph $G$ is said to be \emph{$\mathcal{H}$-free} if $G$ has no induced subgraph isomorphic to $H$ for all $H\in \cH$. 
Tutte \cite{De47, De54} showed that for any $n$, there exists a triangle-free graph with chromatic number at least $n$. (See \cite{Mycielski1955} for another construction, and see \cite{Scott-Seymour2020} for more constructions.) Hence in general there exists no function of $\omega(G)$ that gives an upper bound on $\chi(G)$ for all graphs $G$.
On the other hand, if for some class of graphs there is a function $f$ such that $\chi(G) \leq f(\omega(G))$ for every graph $G$ in this class, we say that this class is \emph{$\chi$-bounded}, with \emph{binding function} $f$. One well-known example is the class of perfect graphs, which has binding function $f(x) = x$, see \cite{crst} for a characterization of perfect graphs. Gy\'arf\'as \cite{Gyarfas1975} and Sumner \cite{Sumner} independently conjectured that for every tree $T$, the class of $T$-free graphs is $\chi$-bounded. This conjecture has been confirmed for some special trees $T$ (see, for example, \cite{CSS2019, Gyarfas1975, GST1980, Kierstead-Penrice1994, Kierstead-Zhu2004, Scott1997, Scott-Seymour2020}), but remains open in general.

We are interested in the following question: For which trees $T$ is the family of $T$-free graphs $\chi$-bounded with binding function $f(x) = x+1$? This question is connected to Vizing's theorem on chromatic index in the sense that for any graph $G$ (with $G\not= K_3$), $\chi(L(G))=\chi'(G) \le  \Delta(G)+1=\omega(L(G))+1$, where $\chi'(G)$ denotes the chromatic index of $G$ and $L(G)$ denotes the line graph of $G$. We see that the class of line graphs is $\chi$-bounded with binding function $f(x) = x+1$. Motivated by this, we say that a class of graphs satisfies the \emph{Vizing bound} if it is $\chi$-bounded with binding function $f(x) = x+1$. 

A pair $\{A, B\}$ of connected graphs is said to be a \emph{Vizing pair} if the class of $\{A, B\}$-free graphs satisfies the Vizing bound. A Vizing pair $\{A,B\}$ is called a \textit{good Vizing pair} if 
neither the condition “$A$-free” nor the condition “$B$-free” is redundant. The study on the characterization of Vizing pairs was initiated by Randerath \cite{Randerath} in 1998. A result of Kierstead \cite{Kierstead1984} in 1984 implied that $\{K_5^-, K_{1,3}\}$ is a good Vizing pair, where $K_5^-$ is obtained from $K_5$ by removing one edge. This generalizes Vizing's theorem, as line graphs are characterized \cite{Beineke1968} as $\cH$-free graphs for some family $\cH$ with $|\cH|= 9$ and $\{K_5^-, K_{1,3}\} \subseteq \cH$.
A good Vizing pair $\{A, B\}$ is said to be \emph{saturated} if for every good
Vizing pair $\{A', B'\}$ with $A \subseteq A'$ and $B \subseteq B'$, we have $A \cong A'$ and $B \cong B'$. 
It is easy to see that if $\{A,B\}$ is a good Vizing pair, then $\{A',B'\}$ is also a Vizing pair for any $A'\subseteq A$ and $B'\subseteq B$. Thus, to characterize all the good Vizing pairs, it suffices to consider the saturated Vizing pairs.

A result by Erd\H{o}s \cite{erdos} implies that for any pair of graphs $\{A,B\}$, in order for $\{A,B\}$-free graphs to have bounded chromatic number, at least one of $A$, $B$ must be acyclic. Therefore, in any saturated Vizing pair $\{A, B\}$, at least one of $A$ or $B$ must be a tree. Randerath \cite{Randerath} studied the first nontrivial case for saturated Vizing pairs, namely pairs of graphs $\{A, B\}$ such that every $\{A, B\}$-free graph $G$ has $\chi(G) \leq 3$. For any such $G$, clearly $K_4 \nsubseteq G$; so, without loss of generality, $A$ must be an induced subgraph of $K_4$, and the only nontrivial cases are $A = K_3$ and $A = K_4$. In either case, $B$ must be a tree. Randerath \cite{Randerath} showed that for $A = K_4$, $\{K_4, P_4\}$ is the only good pair that is saturated. To characterize the good Vizing pairs of the form $\{K_3,B\}$, Randerath \cite{Randerath} proved that $B$ must either be the $H$ graph (the tree with two vertices of degree $3$ and four vertices of degree $1$, see Figure \ref{fig:forbidden_graphs}), or $B$ must be an induced subgraph of the \textit{fork}, the graph obtained from $K_{1,4}$ by subdividing two edges, see Figure \ref{fig:forbidden_graphs}. Randerath further showed that $\{K_3,B\}$ is a good Vizing pair for $B \in \{ H, E, \mbox{ \text{cross}}\}$, where the $E$ graph is obtained from $K_{1,3}$ by subdividing two distinct edges, and the \emph{cross} graph is obtained from $K_{1,4}$ by subdividing one edge, see Figure \ref{fig:forbidden_graphs}.
Observe that both the $E$ graph and the cross graph are induced subgraphs of the fork. 

The only remaining case to complete the characterization of saturated good Vizing pairs ($K_3, B$) is the case when $B$ is the \text{fork} graph (also called a \textit{trident}). Randerath \cite{Randerath} conjectured in 1998 that $\{K_3, \text{fork}\}$ is a good Vizing pair, which is also known as the \textit{trident}-conjecture (see \cite{Schiermeyer-Randerath2019}).

\begin{conjecture}[Randerath \cite{Randerath, Randerath_the_second}]\label{conj:trident}
For any \{$K_3$, fork\}-free graph, $\chi(G) \leq \omega(G)+1 \leq 3$.
\end{conjecture}

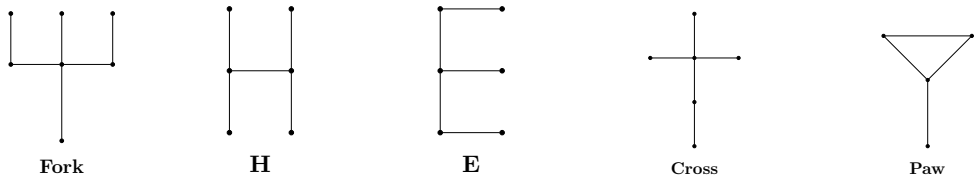
\begin{figure}[htb]
\hbox to \hsize{
	\hspace*{1cm}
	\resizebox{1.5cm}{!}{\begin{tikzpicture}[scale=1, Wvertex/.style={circle, draw=black, fill=white, scale=2}, bvertex/.style={circle, draw=black, fill=black, scale=0.2},rvertex/.style={circle, draw=red, fill=red, scale=0.2}]

\node [bvertex] (v1) at (0,0) {};
\node [bvertex] (v2) at (1,0) {};
\node [bvertex] (v3) at (1,1) {};
\node [bvertex] (v4) at (-1,0) {};
\node [bvertex] (v5) at (-1,1) {};
\node [bvertex] (v6) at (0,1) {};
\node [bvertex] (v7) at (0,-1.5) {};

\draw (v1) -- (v2) -- (v3);
\draw (v1) -- (v4) -- (v5);
\draw (v1) -- (v6);
\draw (v1) -- (v7);

\node[] at (0,-2) {\textbf{Fork}};

\end{tikzpicture}	}%
	\hspace*{1cm}
	\resizebox{1cm}{!}{\begin{tikzpicture}[scale=1, Wvertex/.style={circle, draw=black, fill=white, scale=2}, bvertex/.style={circle, draw=black, fill=black, scale=0.2},rvertex/.style={circle, draw=red, fill=red, scale=0.2}]

\node [bvertex] (v1) at (0.5,1) {};
\node [bvertex] (v2) at (0.5,0) {};
\node [bvertex] (v3) at (0.5,-1) {};
\node [bvertex] (v4) at (-0.5,1) {};
\node [bvertex] (v5) at (-0.5,0) {};
\node [bvertex] (v6) at (-0.5,-1) {};

\draw (v1) -- (v2) -- (v3);
\draw (v4) -- (v5) -- (v6);
\draw (v5) -- (v2);

\node[] at (0,-1.5) {\textbf{H}};
\end{tikzpicture}	}
	\hspace*{1cm}
	\resizebox{1cm}{!}{\begin{tikzpicture}[scale=1, Wvertex/.style={circle, draw=black, fill=white, scale=2}, bvertex/.style={circle, draw=black, fill=black, scale=0.2},rvertex/.style={circle, draw=red, fill=red, scale=0.2}]

\node [bvertex] (v1) at (0.5,1) {};
\node [bvertex] (v2) at (0.5,0) {};
\node [bvertex] (v3) at (0.5,-1) {};
\node [bvertex] (v4) at (-0.5,1) {};
\node [bvertex] (v5) at (-0.5,0) {};
\node [bvertex] (v6) at (-0.5,-1) {};

\draw (v1) -- (v4);
\draw (v2) -- (v5);
\draw (v3) -- (v6);
\draw (v4) -- (v5) -- (v6);

\node[] at (0,-1.5) {\textbf{E}};

\end{tikzpicture}	}
	\hspace*{1cm}
	\resizebox{1.3cm}{!}{\begin{tikzpicture}[scale=1, Wvertex/.style={circle, draw=black, fill=white, scale=2}, bvertex/.style={circle, draw=black, fill=black, scale=0.2},rvertex/.style={circle, draw=red, fill=red, scale=0.2}]

\node [bvertex] (v1) at (0,1) {};
\node [bvertex] (v2) at (0,0) {};
\node [bvertex] (v3) at (0,-1) {};
\node [bvertex] (v4) at (0,-2) {};
\node [bvertex] (v5) at (-1,0) {};
\node [bvertex] (v6) at (1,0) {};

\draw (v1) -- (v2) -- (v3) -- (v4);
\draw (v2) -- (v5);
\draw (v2) -- (v6);

\node[] at (0,-2.5) {\textbf{Cross}};

\end{tikzpicture}	}
	\hspace*{1cm}
	\resizebox{1.3cm}{!}{\begin{tikzpicture}[scale=1, Wvertex/.style={circle, draw=black, fill=white, scale=2}, bvertex/.style={circle, draw=black, fill=black, scale=0.2},rvertex/.style={circle, draw=red, fill=red, scale=0.2}]

\node [bvertex] (v1) at (0,0) {};
\node [bvertex] (v2) at (-1,1) {};
\node [bvertex] (v3) at (1,1) {};
\node [bvertex] (v4) at (0,-1.5) {};

\node[] at (0,-2) {\textbf{Paw}};

\draw (v4) -- (v1) -- (v3) -- (v2) -- (v1);

\end{tikzpicture}	}
	\hspace*{1cm}
}
    \caption{Forbidden graphs}
    \label{fig:forbidden_graphs}
\end{figure}

In 2014, Fan, Xu, Ye, and Yu \cite{yeyu} made a significant progress on this conjecture by showing that the conjecture holds for $C_5$-free graphs. 

\begin{theorem} [Fan, Xu, Ye, and Yu \label{thm:FXYY2014}\cite{yeyu}]\label{thm:grith7}
Let $G$ be a fork-free graph with odd girth at least $7$, then $\chi(G)\leq 3$.
\end{theorem}

In this paper, we confirm Conjecture \ref{conj:trident} in its full generality.

\begin{theorem}\label{thm:main_theorem}
For any $\{K_3, \text{fork}\}$-free graph $G$, $\chi(G) \leq 3$.
\end{theorem}

Thus together with all the previous results, we now have a complete characterization of all the graphs $B$ such that $\{K_3, B\}$-free graphs are $3$-colorable.

The \textit{paw} graph is a triangle with an attached end vertex (see Figure \ref{fig:forbidden_graphs}). Randerath \cite{Randerath} observed that the results concerning triangle-free graphs can be transformed to those concerning paw-free graphs, i.e., \{paw, $B$\} is a good Vizing pair if and only if $\{K_3, B\}$ is a good Vizing pair. Hence Theorem \ref{thm:main_theorem} implies the following corollary.

\begin{corollary}
\{paw, $B$\} is a saturated Vizing pair if and only if  $B\in \{H, \textrm{ fork}\}$.
\end{corollary}

For convenience, a path (respectively, cycle) in a graph is often represented as a sequence (respectively, cyclic sequence) of vertices, with consecutive vertices being adjacent in the graph. 
Throughout the paper, a fork is denoted by a $5$-tuple consisting of vertices and edges. For instance,  $(w_1, w_2 w_3, w_4 w_5, w_6, w_7)$ denotes the fork in which the vertices are labelled such that $w_1$ is the unique vertex of degree $4$ in the fork, and $w_1w_2w_3,
w_1w_4w_5, w_1w_6, w_1w_7$ are paths from $w_1$ to the four leaves of the fork. 
Let $G$ be a graph. 
Given a subgraph $S$ of $G$ and a vertex $v \in V(G)$, let $d(v,S) = \min_{s\in V(S)}\{ d(v,s)\}$, where $d(v,s)$ denotes the distance between $v$ and $s$ in $G$. For a positive integer $i$, let $N^i(S) = \{v \notin S: d(v,S) = i\}$, $N^{\geq i}(S) = \cup_{j\geq i} N^j(S)$, and we often write $N(S)$ for $N^1(S)$.
If $S= \{v\}$, we write $N(v)$ for $N(\{v\})$. 
Two paths are called \textit{independent} if no vertex of any path is internal to the other.

For the rest of this section, we describe additional notations and terminologies used in the paper.
For a positive integer $k$, we use $[k]$ to denote the set $\{1, \ldots, k\}$. Let $G$ be a graph. 
For $u,v\in V(G)$ with $uv\not\in E(G)$, define $G+uv$ as the graph obtained from $G$ by adding $uv$ as an edge.
For $S\subseteq V(G)$, let $G[S]$ denote the subgraph of $G$ induced by $S$ and let $G-S$ denote the graph $G[V(G)\backslash V(S)]$.
For any subgraph $K$ of $G$, let $G-K = G-V(K)$, and if $V(K) = \{v\}$, then we write $G-v$ instead. 
For $a,b\in V(G)$ and $R\subseteq V(G)$, $a \sim b$ denotes that $ab\in E(G)$, and $a \sim R$  denotes that there exists  a vertex $r \in R$ such that $a \sim r$. For $A, B\subseteq V(G)$, we use  $A\sim B$ to denote the case when $a\sim b$ for some $a\in A$ and $b\in B$. Moreover, we say that $A$ is {\it complete} to $B$ if $a\sim b$ for all $a\in A$ and $b\in B$, and that $A$ is {\it anticomplete} to $B$ if $a\nsim b$ for all $a\in A$ and $b\in B$. 
Given a vertex $k$-coloring $\sigma: V(G)\to [k]$ of a graph $G$, we use $\sigma(v)$ to denote the color of a vertex $v$ and $\sigma(S):= \cup_{v\in S} \sigma(v)$ to be the set of colors appearing in a vertex subset $S$. Moreover, given a color $j\in [k]$, we use $\sigma^{-1}(k)$ to denote the set of vertices colored $j$.

\section{Outline of proof techniques}

For the remainder of the paper, let $G$ be a graph such that
\begin{enumerate}[(1)]
    \itemsep0em 
    \item $G$ is \{$K_3$, fork\}-free,
    \item $\chi(G) \geq 4$,
    \item subject to (1) and (2), $|V(G)|$ is minimum.
\end{enumerate} We will show that $G$ does not exist and hence Theorem~\ref{thm:main_theorem} holds. Often, when we label sets and vertices as $A_i$ and $a_i$, the subscripts are taken modulo $5$; so $A_0 = A_5$ and $A_6 = A_1$. Throughout, we will use $1,2,3$ to represent colors in all $3$-colorings of graphs.

Our proof of Theorem~\ref{thm:main_theorem} proceeds in the following manner:

 \begin{itemize}
 \item We show that $G$ is $3$-connected and, for any $5$-cycle $C$ in $G$ and distinct vertices $u_1,u_2\in N^2(C)$, $G$ has no independent paths of length $2$ from both $u_1$ and $u_2$ to the same vertex on $C$. 
 This is proved in Section \ref{sec:independent_paths}. 

  \item We then consider $5$-cycles $C$ in $G$ that minimize $\left|N^{\geq 2}(C)\right|$ (such $5$-cycles are called \textit{good} $5$-cycles). We will derive useful properties about the structure of $G$ around a good $5$-cycle. This is done in Section \ref{sec:good_cycle}.

  \item Using the structural information obtained in Sections \ref{sec:independent_paths} and \ref{sec:good_cycle}, we complete the proof of Theorem~\ref{thm:main_theorem} in Section \ref{sec:fork_main} by breaking the proof into cases based on the number of vertices in $C$ that are of degree $3$ in $G$.
 \end{itemize}

To avoid repetitive technical analysis,
part of our proof will be computer-assisted, whose implementation can be accessed at 
\href{https://github.com/jschroeder35/fork_triangle_free_code}{https://github.com/jschroeder35/fork\_triangle\_free\_code}. 
For a purely combinatorial proof without computer assistance, please refer to an earlier version of our paper at \href{https://arxiv.org/pdf/2111.10469v1}{https://arxiv.org/pdf/2111.10469v1}. In the rest of the section, for the rigour of the paper, we will give a high-level description of the computer search algorithms that are used to derive additional structures or contradictions in the paper. Readers who are not interested in the implementation of these algorithms are welcomed to skip the rest of this section.

There are three main algorithms which we refer to them as \textit{Structure Algorithm}, \textit{Coloring Algorithm}, and \textit{Mixed Algorithm} respectively. These algorithms all start with some substructure of $G$ that is known or assumed at certain stage of the proof and seek to find as much structure or a contradiction if possible. Graphs in the code have an edge-list representation, but since it is not always known whether there is an edge between a pair of vertices or not, each vertex has three lists of vertex pairs: edges, non-edges, or unknown edges incident to it. In the case of forbidden induced subgraphs (such as forks and triangles), note that there are no unknown vertex pairs.  

\begin{itemize}
    \item (Structure Algorithm) This algorithm adds structures or forces a contradiction through the following nested steps: 
    
    \begin{enumerate}
        \item[(1)] Given a forbidden induced subgraphs list, iterate over every unknown vertex pair in the graph and check if, after switching the unknown vertex pair to an edge or a non-edge, any of the forbidden subgraphs appears in the graph. If both cases result in some forbidden subgraph, then (1) returns a contradiction; otherwise if there is a forbidden subgraph when adding an edge, then we switch the unknown vertex pair to a non-edge, and vice versa. This continues until no further updates can be made.

        \item[(2)] Given a forbidden induced subgraphs list, for each of its members $F$ in the list, (2) checks if there exists an induced subgraph $H$ of $G$ such that $E(F) = E(H)$; for each such subgraph $H$, it first returns a list of the unknown vertex pairs in $H$. Among this list of vertex pairs, at least one vertex pair must be an edge to prevent the forbidden graph $F$ from occurring in $G$. Thus for each vertex pair in the list, we temporarily set it to be an edge and apply (1) to the resulting graph. If a contradiction is obtained for each such $H$, then (2) returns a contradiction. Otherwise, any edge/non-edge that is present in all the cases that do not result in a contradiction will be added to the graph, and this process continues until no further updates can be made. Optionally, it will iterate over all $5$-cycles in the graph besides $C_1$. If any $5$-cycle $C$ does not have at least as many vertices of distance $2$ away from it as $C_1$, then the algorithm iterates over all possible ways to add a new vertex which is adjacent to $C_1$ and has distance $2$ from the cycle $C$, and then recursively run (2) up to a specified maximum depth. Again, a contradiction is returned if all cases lead to a contradiction, otherwise edges/non-edges are added to the graph if they are present in all cases that do not lead to a contradiction.

        \item[(3)] Iterate over every unknown vertex pair in $G$, add it to the graph as an edge, then as a non-edge and applying (2) in both cases. If both cases result in contradictions, then (3) returns a contradiction. Otherwise, if changing the unknown vertex pair to an edge results in a contradiction, we then switch it to a non-edge, and update the graph based on (2), and vice versa.

        \item[(4)] Alternate between applying (3), and attempting to add new vertices to the graph that are forced to exist either by the choice of $C_1$ (in Section \ref{sec:independent_paths}) similar to (2), or by neighborhood containment condition, as in Lemma \ref{lem:no_neighborhood_containment}. 
    \end{enumerate}

    \item (Coloring Algorithm) This algorithm takes in a partial coloring of $G$ and a starting vertex $v$, and for each color $i$, find a list of all current or possible neighbors of $v$ that could prevent $v$ from being colored with color $i$ among (a) current or possible neighbors of $v$ that are or could be colored $i$; (b) newly added vertices that are colored $i$; (c) (in recursion) current or possible neighbors of $v$ that would be recolored with $i$ by the recoloring strategy. If this list is empty, this algorithm returns a contradiction as $v$ can be assigned color $i$. If this list contains exactly one vertex, namely $w$, then we add the relevant coloring/edges to the graph to make $w$ a neighbor of $v$ of color $i$, and (if $w$ is not a vertex in case (b) or (c)) recurse this algorithm on $w$. If the list contains more than one vertices, and all cases require updates to the graph, then we check if any edges, non-edges, or vertex colors are common among all cases; and if so, add them to the graph.
    \item (Mixed Algorithm) This algorithm alternates between running the Structure Algorithm and doing the following: it iterates over every triple of vertices $v_1, v_2, v_3$ in the graph, and if $N(v_1) \subseteq N(v_2) \cup N(v_3)$, then the algorithm starts a new coloring $\sigma$ of the graph where $\sigma(v_2) = 1$ and $\sigma(v_3)=2$, runs the Coloring Algorithm above with $v=v_1$ and updates the graph accordingly if no contradiction is obtained (the coloring $\sigma$ is removed after the updates). 
    
    Optionally, it also will check, for each degree-$3$ vertex $v$ in the graph, whether assuming $v$ has no other neighbors and applying the Coloring Algorithm on $v$ (in which $v$'s neighbors must be colored $1$, $2$, and $3$ respectively) produces a contradiction. If the Coloring Algorithm returns a contradiction, then a new neighbor of $v$ is added to the graph before continuing. Alternately, if adding a new neighbor of $v$ produces a contradiction after applying (3) from the Structure Algorithm, then $v$'s max degree is set to be $3$, and any new structure obtained from applying the Coloring Algorithm with starting vertex $v$ is added to the graph.

\end{itemize}

Note that by the minimality of $G$, there must exist a 3-coloring $\sigma$ of $G-v$ for any $v \in G$. When running the (coloring) algorithm, we suppose that the colored vertices are providing partial information about $\sigma$, and $v$ must have neighbors of each color since $G$ is assumed not to be $3$-colorable. We also remark that the lemmas in Section \ref{sec:independent_paths} and Section \ref{sec:good_cycle} all have forbidden subgraph characterizations. Therefore, the structural constraints obtained from these lemmas, once established, are added into the forbidden subgraph list.

\section{Independent paths from $C$ to $N^2(C)$}\label{sec:independent_paths}
Recall that $G$ is a minimum-order graph among all \{$K_3$, fork\}-free graphs with $\chi(G) \geq 4$.
The main objective of this section is to show Lemma \ref{lem:no-2-converge-1}, which states that for any $5$-cycle $C$ in $G$ and distinct vertices $u_1,u_2\in N^2(C)$, $G$ has no independent paths of length $2$ from both $u_1$ and $u_2$ to some common vertex on $C$. 
First, we show an easy fact, which will often be used later without explicit reference.

\begin{lemma}\label{lem:no_neighborhood_containment}
For any two non-adjacent vertices $u, v \in V(G)$, $N(v) \not\subseteq N(u)$.
\end{lemma}

\begin{proof}
Suppose $N(v) \subseteq N(u)$. By the choice of $G$, $G-v$ admits a $3$-coloring, say $\sigma$. But now we can color $v$ with $\sigma(u)$ to extend $\sigma$ to a 3-coloring of $G$, contradicting that $\chi(G)\geq 4$.
\end{proof}

We also need the following result on the connectivity of $G$.

\begin{lemma}\label{lem:no-2-cut}
$G$ is 3-connected.
\end{lemma}

\begin{proof}
Clearly, $G$ must be $2$-connected by the minimality of $G$. Now assume that $G$ is not $3$-connected and let
$\{x,y\}$ be a $2$-cut in $G$. Then $G$ has two induced subgraphs $H_1,H_2$ such that $H_1\cup H_2=G$, $V(H_1\cap H_2)=\{x,y\}$,  and $V(H_i)\setminus V(H_{3-i})\ne \emptyset$ for $i=1,2$. By the choice of $G$, both $H_1$ and $H_2$ admit $3$-colorings. 

If there exist $3$-colorings $\sigma_i$ for $H_i$, $i=1,2$, such that $|\sigma_1(\{x,y\})|=|\sigma_2(\{x,y\})|$, then we can easily modify and combine $\sigma_1$ and $\sigma_2$ to give a $3$-coloring of $G$, a contradiction. Thus, we may assume that for any $3$-colorings $\sigma_i$ of $H_i$, $i=1,2$, we have
$|\sigma_i(\{x,y\})|=i$. This, in particular, implies that $xy\notin E(G)$. 
Without loss of generality, we may assume that $\sigma_1(x)=\sigma_1(y)=\sigma_2(x)=1$ and $\sigma_2(y)=2$. 

First, there exists  $z_2\in N(\{x,y\})\cap V(H_2) \cap \sigma_2^{-1}(3)$; otherwise, we may modify $\sigma_2$ (by changing $\sigma_2(x)$ and $\sigma_2(y)$ to $3$) to obtain a $3$-coloring $\sigma_2'$ of $H_2$ with 
$|\sigma_2'(\{x,y\})|=1$, a contradiction. By symmetry, assume $z_2\in N(x)$.

Next, we claim that there exists a path $u_1xx_1v_1$ in $H_1$ such that $\sigma_1(u_1)=\sigma_1(v_1)=3$ and $\sigma_1(x_1)=2$. First, there exists $u_1\in N(x)\cap V(H_1)\cap \sigma_1^{-1}(3)$; otherwise, we may modify $\sigma_1$ (by changing $\sigma_1(x)$ from $1$ to $3$) to obtain a $3$-coloring $\sigma_1'$ of $H_1$ with $|\sigma_1'(\{x,y\})|=2$, a contradiction. Similarly, there exists $x_1\in N(x)\cap V(H_1)\cap \sigma_1^{-1}(2)$. Moreover, we may choose $x_1$ so that there exists $v_1\in N(x_1)\cap V(H_1) \cap \sigma_1^{-1}(3)$; as, otherwise, we may modify $\sigma_1$ (by recoloring $N(x)\cap V(H_1)\cap \sigma_1^{-1}(2)$ with $3$ and recoloring $x$ with $2$)
to obtain a $3$-coloring $\sigma_1'$ of $H_1$ with $|\sigma_1'(\{x,y\})|=2$, a contradiction. 

Similarly, there exists 
$x_2\in N(x)\cap V(H_2)\cap\sigma_2^{-1}(2)$; for otherwise we can modify $\sigma_2$ to obtain a $3$-coloring $\sigma_2'$ of $H_2$ such that $|\sigma_2'(\{x,y\})|=1$, a contradiction. Moreover, the vertex $x_2$ can be chosen so that there exists $v_2\in N(x_2)\cap V(H_2)\cap \sigma_2^{-1}(3)$. 
Now $(x,x_1v_1, x_2v_2, u_1,z_2)$ is a fork in $G$, a contradiction.
\end{proof}

Before we prove Lemma \ref{lem:no-2-converge-1}, we prove two lemmas concerning $N^2(C)$ for any $5$-cycle $C$ in $G$. We remark again that all proofs in the rest of the paper that use computer assistance also admit a purely combinatorial proof, which can be found in an earlier version of our paper\footnote{https://arxiv.org/pdf/2111.10469v1.}.

\begin{lemma}\label{lem:u1-u2}
Let $C$ be a $5$-cycle in $G$ and let $u_1, u_2 \in N^2(C)$ be distinct. Suppose $u_1 \sim u_2$. Then $N^2(u_1)\cap N^2(u_2)\cap V(C)=\emptyset$.
\end{lemma} 
\begin{proof}
Suppose otherwise that there exists $v_1\in V(C)$ such that $v_1\in N^2(u_1)\cap N^2(u_2)$. Let $C=v_1v_2v_3v_4v_5v_1$, $a_1\in N(u_1)\cap N(v_1)$, and $b_1 \in N(u_2) \cap N(v_1)$. Then $a_1\neq b_1$ to avoid the triangle $a_1 u_1 u_2 a_1$, $v_3\sim \{a_1,b_1\}$ to avoid the fork  $(v_1, v_2 v_3, a_1u_1,
v_5, b_1)$, and $v_4\sim \{a_1, b_1\}$ to avoid the fork $(v_1, v_5 v_4, b_1 u_2, a_1, v_2)$. Since $G$ is triangle-free, we may assume by symmetry that $v_3 \sim b_1$ and $v_4 \sim a_1$. 
Define $$A:=N(u_1)\cap N(v_4), B=N(u_2)\cap N(v_3), \mbox{ and } Z=(N(v_2)\cup N(v_5))\setminus (N(v_3)\cup N(v_4)).$$
Note that $Z\cap (A\cup B)=\emptyset$ by construction, and $A \cap B = \emptyset$ as $G$ is triangle-free. 

\begin{claim}\label{cl:ABZpartition}
 $N(C)=A\cup B\cup Z$. 
 \end{claim}
 \begin{proof}
Let $x \in N(C)$. By symmetry, we may assume that $x\sim \{v_1, v_4, v_5\}$.
First, suppose $x \sim v_1$. Then $x \sim \{u_1, v_3\}$ to avoid the fork $(v_1, v_2 v_3, a_1 u_1, v_5, x)$. If $x \sim u_1$ then $x\nsim u_2$ to avoid the triangle $xu_1u_2x$ and hence $x \sim v_4$ to avoid the fork $(v_1, v_5 v_4, b_1 u_2, v_2, x)$; so $x \in A$. If $x \sim v_3$ then $x\nsim v_4$ to avoid triangle $xv_3v_4x$ and $x \sim u_2$ to avoid the fork $(v_1, v_5 v_4, b_1 u_2, v_2, x)$; so $x \in B$.

Now suppose $x \sim v_4$. Then $x \sim \{u_1, v_2\}$ to avoid the fork $(v_4, v_3 v_2, a_1 u_1, v_5, x)$. If $x \sim u_1$ then $x \in A$; so assume $x \sim v_2$ and $x \nsim u_1$. Then $x\sim b_1$ to avoid the fork $(v_1, v_2 x, a_1 u_1, v_5, b_1)$. But then $(b_1, v_1 v_5, u_2 u_1, v_3, x)$ is a fork in $G$, a contradiction.

Thus, we may assume $x \sim v_5$. If $x\not\sim v_3$ then $x\in Z$. Thus, assume
$x \sim v_3$. If $x\sim u_2$ then $x\in B$. If $x\not\sim u_2$ then $x\sim a_1$ to avoid the fork $(v_3, b_1u_2, v_4a_1, v_2,x)$. Now $(a_1,u_1u_2,v_1v_2, v_4,x)$ is a fork in $G$, a contradiction.  \end{proof}

\begin{claim}\label{cl:no_dist2_adj_AB}
$N^2(C) \cap (N(A) \cup N(B)) = \{u_1, u_2\}$, and $N^2(C) \subseteq N(u_1)\cup N(u_2)$.
\end{claim}
\begin{proof}
 Let $u \in N^2(C) \setminus \{u_1, u_2\}$. It suffices to show that $u\in N(\{u_1, u_2\})$.
 The Mixed Algorithm applied to $C \cup \{a_1, b_1, u_1, u_2, u\}$ finds a contradiction if $u \sim a_1$, so $u \nsim a_1$, and by symmetry $u \nsim b_1$. Furthermore, the Mixed Algorithm applied to $C \cup \{a_1, b_1, u_1, u_2, u, a\}$ finds a contradiction if $u$ is adjacent to some vertex $a \in N(v_4)$, so $u \nsim A$, and by symmetry $u \nsim B$. Finally, suppose for contradiction that $u \nsim \{u_1, u_2\}$. Then, by Claim \ref{cl:ABZpartition}, $u$ is adjacent to some vertex $z \in Z$. Without loss of generality, suppose $z \sim v_2$. Then, the Mixed Algorithm applied to $C \cup \{a_1, b_1, u_1, u_2, u, z_2\}$ obtains a contradiction.
\end{proof}

\begin{claim}\label{cl:3rd_neighborhood_empty_lem}
$N^3(C) = \emptyset$.
\end{claim}
\begin{proof}
Suppose $N^3(C)\ne \emptyset$ and let $w \in N^3(C)$ be arbitrary. Note that $w \nsim u_1$, otherwise the Mixed Algorithm applied to $C \cup \{a_1, b_1, u_1, u_2, w\}$ obtains a contradiction; similarly $w \nsim u_2$. Suppose now $w \sim u$ for some $u \notin \{u_1, u_2\}$. Then, By Claim \ref{cl:no_dist2_adj_AB}, $u$ must be adjacent to some vertex in $N(Z)$. Without loss of generality suppose that $u \sim N(v_5)$. But then, the Mixed Algorithm applied to $C \cup N(v_5) \cup \{a_1, b_1, u_1, u_2, w, u\}$ obtains a contradiction.
\end{proof}

 Define  $\sigma: V(G)\to [3]$ such that 
 $\sigma((\{v_3, v_5\} \cup N(u_1)) \setminus N(v_5))=\{2\}$, $\sigma((\{v_2, v_4\} \cup N(u_2))\setminus N(v_2))=\{3\}$, and $\sigma(x)=1$ for all $x\in \{v_1\} \cup Z$ not already colored. Since $\sigma$ cannot be a 3-coloring of $G$, there exist $z_1 \in N(v_5)\backslash N(v_2)$ and $z_2 \in N(v_2)\backslash N(v_5)$ such that $z_1$ and $z_2$ are both colored $1$, and $z_1\sim z_2$. But now, the Structure Algorithm applied to $C \cup \{a_1, b_1, u_1, u_2, z_1, z_2\}$ obtains a contradiction.
 This completes the proof of the lemma.
\end{proof}

\begin{lemma}\label{lem:u1-u2-complete}
Let $C$ be any $5$-cycle in $G$, let $u_1, u_2 \in N^2(C)$ with $u_1 \nsim u_2$, and let $v\in V(C)$. Suppose $u_1 a_1 v$ and $u_2 b_1 v$ are independent paths in $G$.  Then $a_1 \sim u_2$ and $b_1\sim u_1$. 
\end{lemma}
\setcounter{claim}{0}

\begin{proof}
Let $C=v_1v_2v_3v_4v_5v_1$ with $v_1=v$. Suppose for a contradiction that $a_1\not\sim u_2$ (the case $b_1\not\sim u_1$ is symmetric).
Then $u_1 \sim b_1$ to avoid the fork $(v_1, a_1 u_1, b_1 u_2, v_2, v_5)$. Moreover,  $v_3 \sim \{a_1, b_1\}$ to avoid the fork $(v_1, v_2 v_3, b_1 u_2, v_5, a_1)$, and  $v_4 \sim \{a_1, b_1\}$ to avoid the fork $(v_1,v_5v_4,b_1u_2, v_2, a_1)$. Since $G$ is triangle-free, we may assume by symmetry that $v_3 \sim b_1$ and $v_4 \sim a_1$. By Lemma \ref{lem:no_neighborhood_containment}, there exists $b_2 \in N(u_2)\backslash N(u_1)$. Now, the Mixed Algorithm applied to $C \cup \{a_1, b_1, u_1, u_2\}$ obtains a contradiction.
\end{proof}

We are now ready to prove the main result of this section.

\setcounter{claim}{0}

\begin{lemma}\label{lem:no-2-converge-1}
Let $C$ be any $5$-cycle in $G$ and let $u_1, u_2 \in N^2(C)$ ($u_1\neq u_2$) and $v\in V(C)$. Then $G$ does not contain two independent paths of length 2 from $v$ to $u_1,u_2$, respectively. 
\end{lemma}
\begin{proof}
Let $C=v_1v_2v_3v_4v_5v_1$ and without loss of generality suppose $G$ has independent paths $v_1 a_1 u_1$ and $v_1 b_1 u_2$. By Lemma \ref{lem:u1-u2}, $u_1\nsim u_2$.
By Lemma \ref{lem:u1-u2-complete}, $a_1\sim u_2$ and $b_1\sim u_1$.
By Lemma \ref{lem:no_neighborhood_containment}, 
let $x_1 \in N(u_1)\backslash N(u_2)$ and $x_2 \in N(u_2) \backslash N(u_1)$. 
Now, the Structure Algorithm applied to $C \cup \{a_1, b_1, u_1, u_2, x_1, x_2\}$ obtains the following structure: $\{a_1, b_1\} \nsim \{v_3, v_4\}$, there exist new vertices $c_1 \in N(a_1) \setminus N(b_1)$ and $c_2 \in N(b_1) \setminus N(a_1)$, $\{c_1, c_2\} \nsim \{v_3, v_4\}$, and $c_1 \nsim c_2$. Let $H = C \cup \{a_1, b_1, u_1, u_2, c_1, c_2, x_1, x_2\}$. Note that $c_1 \sim \{v_2, v_5\}$ to avoid the fork $(v_1, a_1c_1, v_2v_3, v_5, b_1)$, and by symmetry, so is $c_2$. In the case that $c_1$ is adjacent to both $v_2$ and $v_5$, then the Mixed Algorithm applied to $H$ obtains a contradiction, so $c_1$ is adjacent to exactly one vertex in $\{v_2, v_5\}$, and by symmetry so is $c_2$. Without loss of generality, $c_1 \sim v_5$ and $c_1 \nsim v_2$. Now, $c_2 \sim v_2$ or $c_2 \sim v_5$, and in either case, the Mixed Algorithm applied to $H$ obtains a contradiction.
\end{proof}

\section{Properties about good $5$-cycles}\label{sec:good_cycle}
Recall that a $5$-cycle $C$ in $G$ is a good $5$-cycle if $|N^{\geq 2}(C)|$ is minimized. In this section, we prove a few useful properties about the structure of $G$ surrounding a good $5$-cycle $C$. In particular, we will show that no vertex of $G - C$ has two neighbors on $C$. We first need a technical lemma on the second neighborhood of a good $5$-cycle. 
\setcounter{claim}{0}

\begin{lemma}\label{lem:2nd-neighborhood-independent}
Let $C$ be a good 5-cycle in $G$. If $N^3(C)=\emptyset$ then $E(G[N^2(C)])\ne \emptyset$.  
\end{lemma}

\begin{proof}
Suppose $N^3(C)=\emptyset$ and $E(G[N^2(C)])=\emptyset$. For $i\in [5]$, let $A_i = N(C)\setminus \cup_{j\in [5]\setminus \{i\}}N(v_j)$ and $B_i = N(v_{i-1})\cap N(v_{i+1})\cap N(C)$, where subscripts are taken modulo $5$. Then $N(C) = \cup_{i\in [5]} (A_i\cup B_i)$.

\begin{claim}\label{z_complete_to_stuff} For any $z \in N^2(C)$ and any $i\in [5]$, if $z \sim A_i$ then $z$ is complete to $B_{i-1}\cup A_i\cup B_{i+1}$, and if $z \sim B_i$ then $z$ is complete to $A_{i-1}\cup B_i\cup A_{i+1}$, where subscripts are taken modulo $5$. \end{claim}

\begin{proof}
Suppose $z \sim a_i$ for some $a_i\in A_i$ and let $c \in B_{i-1}\cup (A_i\setminus \{a_i\}) \cup B_{i+1}$ be arbitrary. Then, $c \sim z$ to avoid the fork $(v_i, a_iz, v_{i+1}v_{i+2}, v_{i-1}, c)$ (when $c\in B_{i-1}$) or $(v_i, a_iz, v_{i-1}v_{i-2}, v_{i+1}, c)$ (when $c \in A_i\cup B_{i+1}$). Thus,  $z$ is complete to $B_{i-1}\cup A_i\cup B_{i+1}$. 

Now suppose $z \sim b_i$ for some $b_i\in B_i$. Let $c \in A_{i-1}\cup B_i \cup A_{i+1}$ be arbitrary. Then, $c \sim z$ to avoid the fork $(v_{i-1}, b_iz, v_{i-2}v_{i-3}, v_{i}, c)$ (when $c\in A_{i-1}$) or $(v_{i+1}, b_iz, v_{i+2}v_{i+3}, v_{i}, c)$ (when $c\in B_i\cup A_{i+1})$. So $z$ is complete to $A_{i-1}\cup B_i\cup A_{i+1}$.
\end{proof}

\begin{claim}\label{A-Complete}
For any $i\in [5]$, if $A_i\sim A_{i+2}$ then $B_{i+1}=\emptyset$ and $A_i$ is complete to $A_{i+2}$. 
\end{claim}
\begin{proof}
Suppose $A_i\sim A_{i+2}$. If $A_i$ is not complete to $A_{i+2}$ then, by symmetry,  let $a_i,a_i' \in A_i$ and $a_{i+2} \in A_{i+2}$ such that $a_i \sim a_{i+2}$ and $a_i'\not\sim a_{i+2}$; now $(v_i, a_ia_{i+2}, v_{i-1}v_{i-2}, v_{i+1}, a_i')$ is a fork in $G$, a contradiction. 
Now if there exists $b_{i+1}\in B_{i+1}$ then, for any $a_i\in A_i$ and $a_{i+2}\in A_{i+2}$,  $(v_i, a_ia_{i+2}, v_{i-1}v_{i-2}, v_{i+1}, b_{i+1})$ is a fork in $G$, a contradiction. 
\end{proof}
   
   \begin{claim}\label{A4A1A3}
  For any $i\in [5]$ with $A_i\sim A_{i+2}$, $A_i\nsim A_{i-2}$ and $A_{i+2}\nsim A_{i+4}$. 
  \end{claim}
  \begin{proof}
  For, suppose the claim fails. By symmetry, we may assume  $A_1\sim A_3$ and $A_1\sim A_4$. Then by Claim \ref{A-Complete}, $B_2=B_5=\emptyset$ and $A_1$ is complete to $A_3\cup A_4$. Note that $A_1 \nsim A_5$; for, otherwise, letting $a_i \in A_i$ for $i\in [5]\setminus \{2\}$ with $a_1\sim a_5$, we see that $(a_1, a_4v_4, v_1v_2, a_5, a_3)$ would be a fork in $G$. Similarly, $A_1 \nsim A_2$. Next, $A_4$ is complete to $B_4$; for
  letting $b_4\in B_4$ and $a_i\in A_i$ for $i=3,4$ with $b_4\nsim a_4$, we see that $(v_3,  v_2v_1, v_4a_4, b_4, a_3)$ would be a fork in $G$. Similarly, $A_3$ is complete to $B_3$. Thus, $A_1\nsim B_3\cup B_4$ as $G$ is triangle-free. 
    
    Define $\sigma$ to be the $3$-coloring of $G[C \cup N(C)]$ such that  $\sigma(N(v_2))=\{1\}$, $\sigma((N(v_5) \setminus N(v_2))\cup A_1\cup \{v_2\})=\{2\}$,  and $\sigma(A_3\cup A_4\cup \{v_5\})=\{3\}$. Now $N^2(C) \ne \emptyset$, or else $\sigma$ would be a $3$-coloring of $G$. Let $Z:=\{z\in N^2(C): |\sigma(N(z)|=3\}$. Then $Z\ne \emptyset$; or else,  since $E(G[N^2(C)])=\emptyset$ and $N^3(C)=\emptyset$, we can color $N^2(C)$ greedily to extend $\sigma$ to a $3$-coloring of $G$, a contradiction. 

    Now let $z\in Z$ be arbitrary. Let $a,b,c\in N(Z)$ such that $\sigma(a)=1$, $\sigma(b)=2$, and $\sigma(c)=3$. Then $a\in N(v_2)$, $b\in (N(v_5)\setminus N(v_2))\cup A_1$, and $c\in A_3\cup A_4$. Note that $b\notin A_1$ as $c$ is complete to $A_3 \cup A_4$ and $G$ is triangle-free. Moreover, $b \notin B_4$, or else $(b, a_4a_1, v_3v_2, v_5, z)$ is a fork in $G$, a contradiction. Hence, $b \in A_5$, and by Claim \ref{z_complete_to_stuff}, $B_4 = \emptyset$.
    
  Since $b\in N(z)\cap \sigma^{-1}(2)$ is arbitrary, it follows from Claim \ref{z_complete_to_stuff} that $\sigma^{-1}(2)\cap N(Z)=A_5$. If $A_5\nsim N(v_2) \cap N(C)$ then 
    we obtain a contradiction by modifying $\sigma$ to a $3$-coloring of $G$: recolor $A_5$ with $1$, color $Z$ with $2$, and color $N^2(C)\setminus Z$ greedily.
    
    Hence, $A_5\sim N(v_2) \cap N(C)$. Since $z, a, b,c$ are chosen arbitrarily, we may assume 
    that $b \sim a'$ for some $a' \in N(v_2) \cap N(C)$. 
    Let $a_i\in A_i$ for $i\in \{1,3,4\}$.  Note that $z\sim a$ but $z\not\sim a'$ (otherwise $zba'z$ is a triangle). By Claim \ref{z_complete_to_stuff}, $a' \in B_3$. Moreover, $a'\sim a_3$ as $B_3$ is complete to $A_3$. Then $b\sim a_4$ to avoid the fork $(a', v_4 a_4, v_2 v_1, b, a_3)$. Now $(b, a' v_2, a_4 a_1, z, v_5)$ is a fork, a contradiction.
    \end{proof}
    
\begin{claim}\label{Ai-nsim-Ai+2}
$A_i \nsim A_{i+2}$ for all $i \in [5]$.
\end{claim}
\begin{proof}
Suppose otherwise. Without loss of generality we may assume that $A_4\sim A_1$. Then by Claim \ref{A4A1A3}, $A_1 \nsim A_3$ and $A_4 \nsim A_2$.
By Claim \ref{A-Complete}, $B_5=\emptyset$ and $A_4$ is complete to $A_1$. Define $\sigma$ to be the $3$-coloring of $G[C \cup N(C)]$ such that  $\sigma(N(v_5))=\{1\}$, $\sigma(\{v_2\}\cup A_1\cup A_3\cup B_2)=\{2\}$, and 
  $\sigma(N(v_4) \cup A_2)=\{3\}$. Let $Z=\{z\in N^2(C): |\sigma(N(z))|=3\}$; then $Z \ne \emptyset$, or else we could extend $\sigma$ greedily to a $3$-coloring of $G$. Let $z\in Z$ be arbitrary and let $c_5 \in N(z)\cap \sigma^{-1}(1)$; so  $c_5\in N(v_5)$.

  First, suppose $z \nsim  A_1\cup A_4$.  Let $b \in N(z) \cap \sigma^{-1}(2)$ and $c \in N(z) \cap \sigma^{-1}(3)$; then, by Claim \ref{z_complete_to_stuff}, $b \in A_3$ and $B_2 = \emptyset$, and similarly $c \in A_2$ and $B_3 = \emptyset$. Hence, $z$ is complete to $A_2\cup A_3$. By Claim \ref{z_complete_to_stuff} and since $1\in \sigma(N(z))$, we have that $z$ is complete to $B_1\cup B_4 \cup A_5$. Since $G$ is triangle-free,  $A_2\cup A_3\nsim A_5 \cup B_1\cup B_4$. So 
  $1 \notin \sigma(N(A_3))\cup \sigma(N(A_2))$. But now we derive a contradiction by modifying $\sigma$ to obtain a $3$-coloring of $G$ by recoloring all vertices in $A_2\cup A_3$ with $1$, and greedily coloring all vertices in $N^2(C)$.

  Thus, $z \sim A_1\cup A_4$. By symmetry, we may assume that $z\sim A_1$. So $z$ is complete to $A_1\cup B_2$ by Claim \ref{z_complete_to_stuff}. Then $z\nsim A_4$, as $G$ is triangle-free. Moreover, $z\nsim B_3$ to avoid contradicting Claim \ref{z_complete_to_stuff}. So $z\sim A_2$ as $\sigma^{-1}(3)\cap N(z)\ne \emptyset$. Then $z$ is complete to $B_1\cup A_2\cup B_3$ by Claim \ref{z_complete_to_stuff}. This implies $B_3=\emptyset$.  Furthermore,  $B_2 = \emptyset$, or else, for any $a_2\in A_2$ and $b_2 \in B_2$, $(z, a_2v_2, c_5v_5, a_1, b_2)$ (when $c_5 \nsim v_2$) or $(z, b_2v_3, c_5v_5, a_1, a_2)$ (when $c_5 \sim v_2$) is a fork in $G$, a contradiction. Also, $B_1 = \emptyset$; or else for any $b_1 \in B_1$, $b_1 \sim z$ by Claim \ref{z_complete_to_stuff}, and hence $(v_2, v_3v_4, v_1a_1, b_1, a_2)$ is a fork in $G$, a contradiction.

  Note that  $1 \notin \sigma(N(A_2))$. For, let $a_2\in A_2$ and $a_2' \in N(a_2)\cap \sigma^{-1}(1)$. Then $a_2'\ne c_5$ and $a_2'\nsim z$, as $z$ is complete to $A_2$ and $G$ is triangle-free. Now $a_2' \sim v_5$ by definition of $\sigma$; so $(v_5, c_5z, v_1v_2, a_2', v_4)$ is a fork in $G$. 
  
  Therefore, we can modify $\sigma$ to a $3$-coloring $\tau$ of $G[C \cup N(C)]$ by letting $\tau(v) = \sigma(v)$ for $v \notin A_2$ and $\tau(A_2) = \{1\}$. Let $Y=\{y\in N^2(C): |\tau(y)|=3\}$. Then $Y\ne \emptyset$ as otherwise $\tau$ can be extended to give a $3$-coloring of $G$. Let $y\in Y$ be arbitrary. Recall that $B_1= B_2 = B_3 = B_5 = \emptyset$. Thus $y \sim A_4=\tau^{-1}(3)\cap N(C)$; hence $y\sim A_3$ since $2\in \sigma(N(y))$ and $y \nsim A_1$. It follows by Claim \ref{z_complete_to_stuff} that $y$ is complete to $A_4\cup A_3\cup B_4$. Now, $B_4 = \emptyset$; or else for any $b_4\in B_4$, $(v_3, v_4a_4, v_2v_1, b_4, a_3)$ would be a fork in $G$. Hence, we have shown that $B_i = \emptyset$ for all $i \in [5]$. Further note that $N(A_4) \cap \tau^{-1}(1)=\{v_4\}$; for otherwise, there exist $c_4 \in A_4$ and $c_4' \in (N(c_4)\backslash V(C)) \cap \tau^{-1}(1)$, and  $(c_4, v_4v_3, a_1v_1, c_4', y)$ is a fork in $G$. But then, we obtain a 3-coloring $\tau'$ of $G$ by letting $\tau'(v) = \tau(v)$ for $v \notin \{v_4\} \cup A_4 \cup N^2(C)$, and  $\tau'(v_4) = 2$, $\tau'(A_4) = \{1\}$, and $\tau'(N^2(C)) = \{3\}$, a contradiction.
\end{proof}     

   \begin{claim}\label{AiNotEmpty}
     $A_i \ne  \emptyset$ for all $i \in [5]$.
   \end{claim}
   \begin{proof}
   For, suppose by symmetry that $A_5=\emptyset$. Then $N^2(C) \nsim B_5$ by the minimality of $\left|N^{\geq 2}(C)\right|$. Let $\sigma$ be the $3$-coloring of $G[C \cup N(C)]$ such that $\sigma(N(v_2))=\{1\}$, $\sigma(N(v_3))=\{2\}$, and $\sigma(\{v_5\}\cup A_1\cup A_4\cup B_5)=\{3\}$. Let $Z:=\{z\in N^2(C): |\sigma(N(z))|=3\}$. Then $Z\ne \emptyset$, or else $\sigma$ can be extended greedily to a coloring of $G$. Let $z\in Z$ be arbitrary. Since $|\sigma(N(z))| = 3$, $z \sim N(v_2)$ and $z\sim N(v_3)$. Let $c_2 \in N(v_2)\cap N(z)$ and $c_3 \in N(v_3)\cap N(z)$. Moreover, $z\sim A_1\cup A_4$ since $Z\nsim B_5$. Note the symmetry between $A_1$ and $A_4$.

    We claim that if $z\sim A_4$, then $1\not\in \sigma(N(A_4))$.
    For, suppose $z\sim A_4$ (so $z$ is complete to $A_4$ by Claim \ref{z_complete_to_stuff}) and $1\in \sigma(N(A_4))$. Let $a_4\in A_4$ and $a_4'\in N(a_4)\cap \sigma^{-1}(1)$. Then $a_4' \in N(v_2)$ by definition and $a_4'\ne c_2$ as $G$ is triangle-free. Now $c_2 \sim v_4$ to avoid the fork $(v_2, c_2z, v_3v_4, a_4', v_1)$. This implies $c_2\in B_3$ and, hence, $z$ is complete  to $B_3$. If $A_3=\emptyset$ then for the 5-cycle $C' = v_5v_1v_2c_2v_4v_5$, $N^{\geq 2}(C') \subsetneq N^{\geq 2}(C)$ (since $A_3 = \emptyset$), contradicting the choice of $C$. So $A_3\ne \emptyset$. Since $c_3 \in N(v_3)\cap N(Z)$, $c_3 \in B_2\cup A_3\cup B_4$. By Claim \ref{z_complete_to_stuff}, $z$ is complete to $A_3$. Now let $a_3\in A_3$. Then, $(v_4, v_5v_1, v_3a_3, a_4, c_2)$ is a fork in $G$, a contradiction.

    Similarly, if $Z\sim A_1$, then $2\not\in \sigma(N(A_1))$. We can now modify $\sigma$ to obtain a $3$-coloring of $G$ by re-coloring $N(Z)\cap A_4$ to $1$, re-coloring $N(Z)\cap A_1$ to $2$, coloring $Z$ with $3$, and coloring $N^2(C)\backslash Z$ greedily.
  \end{proof}

  Now we are ready to complete the proof.
   Let $\sigma$ be the $3$-coloring of $G[C \cup N(C)]-A_5$ such that $\sigma(N(v_2))=\{1\}$, $\sigma(N(v_3))=\{2\}$, and $\sigma(\{v_5\}\cup A_1\cup A_4\cup B_5)=\{3\}$.
   
  \medskip          
        {\it Case} 1. $N^2(C) = \emptyset$.
        
        Let $A_5' = A_5\cap N(B_2)\cap N(B_3)$. 
        Extend $\sigma$ to a $3$-coloring of $G \setminus A_5'$ by letting $\sigma((A_5 \backslash A_5')\cap N(B_2))=\{1\}$ and $\sigma((A_5 \backslash A_5')\backslash N(B_2)) = \{2\}$.
        Then $A_5' \ne \emptyset$ as, otherwise, $\sigma$ would be a $3$-coloring of $G$. 
        Let $B_2' = B_2 \cap N(A_5')$ and $B_3' = B_3 \cap N(A_5')$. Since $G$ is triangle-free, $B_2'$ is anticomplete to $B_3'$, $N(B_2')\cap \sigma^{-1}(3) \subseteq A_4$, and $N(B_3')\cap \sigma^{-1}(3)\subseteq A_1$.  
        Let $B_2'' = B_2'\cap  N(A_4)$, $B_3'' = B_3' \cap N(A_1)$, and let $A_5'' = A_5 \cap  N(B_2'') \cap N(B_3'')$. 
        
        We further modify $\sigma$ to a $3$-coloring of $G \setminus A_5''$ by letting $\sigma(B_2' \setminus B_2'')=\sigma(B_3' \setminus B_3'')=\{3\}$,   $\sigma((A_5' \setminus A_5'')\cap N(B_2''))=\{1\}$, and $\sigma((A_5' \setminus A_5'')\backslash N(B_2'')) = \{2\}$. 
        Then $A_5'' \neq \emptyset$, or $\sigma$ would be a $3$-coloring of $G$. 
        Hence, there exists $a_5 \in A_5''$, $b_2 \in B_2''$, $b_3 \in B_3''$, $a_1 \in A_1$, and $a_4 \in A_4$, such that $a_5 \sim b_2$, $a_5 \sim b_3$, $b_2 \sim a_4$, and $b_3 \sim a_1$. Then $A_5'' \nsim \{a_1, a_4\}$ as $G$ is triangle-free. 

        Let $\tau$ be the $3$-coloring of $G \setminus B_5$ such that  $\tau(N(v_5))=\{1\}$, $\tau(\{v_5, v_2\}\cup B_2\cup A_1\cup A_3)=\{2\}$, and $\tau(\{v_3\}\cup B_3\cup A_2\cup A_4)=\{3\}$. Then, $B_5 \neq \emptyset$, or else $\tau$ would be a $3$-coloring of $G$. Let $b_5 \in B_5$. Then $b_5 \sim a_5$ to avoid the fork $(v_1, v_2v_3, v_5a_5, b_5, a_1)$. But now $(a_5,  b_2a_4, b_3a_1, b_5, v_5)$ is a fork in $G$, a contradiction. 
  \medskip

        {\it Case} 2:   $N^2(C) \neq \emptyset$. 

        First, we show that there exists $i \in [5]$ such that $B_i = B_{i+1}=\emptyset$. Suppose otherwise. Then, without loss of generality we may assume that $B_1$, $B_3$ and $B_5$ are all nonempty. Let $W = A_5 \cup B_1 \cup A_2 \cup B_3 \cup A_4 \cup B_5 \cup A_1$. Then $N^2(C)\subseteq N(W)$; otherwise for any $z\in N^2(C)\setminus N(W)$, we have $N(z)\subseteq N(v_3)$, contradicting Lemma \ref{lem:no_neighborhood_containment}. Thus, by Claim \ref{z_complete_to_stuff}, $N^2(C)$ is complete to $W$. Let $z\in N^2(C)$ and $a_i\in A_i$ for $i=1,2,4,5$. Then $(z, a_5v_5, a_2v_2, a_1, a_4)$ is a fork in $G$, a contradiction.

        Hence without loss of generality, assume $B_2 = B_3 = \emptyset$. Next we show $B_5=\emptyset$.  Suppose $B_5\neq \emptyset$. Let $\sigma$ be a $3$-coloring of $G[C \cup N(C)]$ such that $\sigma(N(v_2) \cup A_5)=\{1\}$, $\sigma(N(v_3))=\{2\}$, and $\sigma(\{v_5\}\cup B_5\cup A_1 \cup A_4)=\{3\}$. There exists $z \in N^2(C)$ with $\left|\sigma\left(N(z)\right)\right| = 3$, or else we can greedily color $N^2(C)$ to extend $\sigma$ to a $3$-coloring of $G$. Hence, there exists  $z\in N^2(C)\cap N(W')$, where $W'=A_4\cup B_5\cup A_1$. By Claim \ref{z_complete_to_stuff}, since $B_5\neq \emptyset$, $z$ is complete to $W'$. Let $a_i\in A_i$ for $i\in \{1,4\}$. Let $c\in N(z)\cap \sigma^{-1}(1)$ and  $c'\in N(z)\cap \sigma^{-1}(2)$. Now $(z, a_1v_1, a_4v_4, c, c')$ is a fork in $G$, a contradiction. 
          
        So $B_5 = \emptyset$. Then by symmetry and by the previous cases, either $B_1 \neq \emptyset$ and $B_4 \neq \emptyset$, or $B_1 = B_4 = \emptyset$.  Let $\sigma$ be a $3$-coloring of $G[C \cup N(C)]$ such that $\sigma(N(v_2) \cup A_5)=\{1\}$, $\sigma(N(v_3))=\{2\}$, and $\sigma(\{v_5\}\cup A_1 \cup A_4)=\{3\}$. Let  $z \in N^2(C)$ with $\left|\sigma\left(N(z)\right)\right| = 3$.
  
         Suppose $B_1 \neq \emptyset$ and $B_4 \neq \emptyset$. 
         Then $z\sim N(C)\cap \sigma^{-1}(1)\subseteq W :=A_3 \cup B_4 \cup A_5 \cup B_1 \cup A_2$. So by Claim \ref{z_complete_to_stuff}, $z$ is complete to $W$. Note that $z\sim N(C)\cap \sigma^{-1}(3)=A_1\cup A_4$; so let $a_i\in N(z)\cap A_i$ for some $i\in \{1,4\}$.   Now $(z, a_2v_2, a_5v_5, a_3, a_i)$ is a fork in $G$, a contradiction. 
            
        Now assume $B_1 = B_4 = \emptyset$. Since $|\sigma(N(z))|=3$, $|\{i: z \sim A_i\}|\ge 3$; so 
        $z \sim A_j$ and $z \sim A_{j+1}$ for some $j$. Without loss of generality, we may assume $z\sim A_5$ and $z\sim A_1$. However, we obtain a contradiction by a $3$-coloring $G$ as follows: coloring $N(v_3) \cup A_5 \cup A_1$ with $1$, coloring $N(v_2) \cup A_4$ with $2$, and coloring $\{v_5\} \cup N^2(C)$ with $3$. 
\end{proof}

\begin{lemma}\label{lem:no-degree2-both-sides}
Let $C$ be a good $5$-cycle in $G$. Then for any $x\in N(C)$, $|N(x) \cap V(C)|\le 1$ or $|N(x) \cap N^2(C)|\le 1$.
\end{lemma}
\begin{proof}
Let $C=v_1v_2v_3v_4v_5v_1$. 
Suppose without loss of generality that there exists some $x \in N(C)$ such that $v_2, v_5 \in N(x)$ and $u_1, u_2 \in N(x) \cap N^2(C)$. By Lemma \ref{lem:no_neighborhood_containment}, there exists a vertex $y_1 \in N(u_1)\backslash N(u_2)$. By Lemma \ref{lem:no-2-converge-1}, $y_1 \not\in N(\{v_2, v_5\})$. Now $y_1 \sim v_3$ to avoid the fork $(x, v_2 v_3, u_1 y_1, v_5, u_2)$, but now $(x, v_5 v_4, u_1 y_1, u_2, v_2)$ is a fork in $G$, a contradiction.
\end{proof}

\begin{lemma}\label{lem:no-partial-diamond}
Let $C$ be a good $5$-cycle in $G$, and 
let $u_1 \in N^2(C)$ and $v\in V(C)$. Then, for any $a_1,b_1\in N(v)\setminus V(C)$, $|N(u_1)\cap \{a_1,b_1\}|\ne 1$.
\end{lemma}
\begin{proof}
Let $C=v_1v_2v_3v_4v_5v_1$. Suppose the assertion is false and assume without loss of generality that
there exist $a_1,b_1\in N(v_1)\setminus V(C)$ and $u_1\in N^2(C)$ such that $a_1 \sim u_1$ but $b_1 \not\sim u_1$. Note that $v_3 \sim \{a_1, b_1\}$ to avoid the fork $(v_1, v_2v_3, a_1 u_1, v_5, b_1)$ and, by symmetry, that $v_4 \sim \{a_1, b_1\}$. Since $G$ is triangle-free, we may assume by symmetry that $a_1 \sim v_4$ and $b_1 \sim v_3$. Now, the Mixed Algorithm applied to $C \cup \{a_1, b_1, u_1\}$ obtains a contradiction.
    \end{proof}

\setcounter{claim}{0}

\begin{lemma}\label{lem:diamond-2-attachment}
Let $C$ be a good $5$-cycle in $G$. 
Let $u_1 \in N^2(C)$ and $v\in V(C)$, and let $a_1, b_1 \in N(u_1) \cap N(v)$ be distinct. 
Then $\{a_1,b_1\}\nsim V(C)\setminus \{v\}$.
\end{lemma}

\begin{proof} Let $C=v_1v_2v_3v_4v_5v_1$ and without loss of generality assume that $v=v_1$. Then $\{a_1,b_1\}\nsim \{v_2,v_5\}$ as $G$ is triangle-free. Suppose the assertion is false and assume, without loss of generality, that $b_1\sim v_3$. Note that there exists some $u_2 \in N^2(C) \setminus \{u_1\}$, otherwise $N^2(C) = \{u_1\}$ and $N^3(C) \neq \emptyset$ by Lemma \ref{lem:2nd-neighborhood-independent}, contradicting that $G$ is $3$-connected (by Lemma \ref{lem:no-2-cut}) as $u_1$ is a cut vertex.

\begin{claim}\label{a_1-not-sim-v_4}
$a_1 \nsim v_4$. 
\end{claim}

\begin{proof}
For the sake of contradiction suppose that $a_1\sim v_4$. Then,  by Lemma \ref{lem:no-2-converge-1}, $u_2 \notin N^2(\{v_1, v_3, v_4\})$; so without loss of generality, there exists $c_2\in N(u_2)\cap N(v_2)$. 

Suppose $N^2(C)\cap N^2(v_5)=\emptyset$. By Lemma \ref{lem:no-partial-diamond}, $N(v_2)\setminus V(C) \subseteq N(u_2)$; so by Lemma \ref{lem:no-2-converge-1}, $d(v_2)=3$ or $N^2(v_2) \cap N^2(C)=\{u_2\}$. 
If $d(v_2) = 3$, then $\{c_2, u_1\}$ is a $2$-cut in $G$, contradicting that $G$ is $3$-connected. Hence $d(v_2) \geq 4$ and $N^2(v_2) \cap N^2(C)=\{u_2\}$. Then $N^3(C) = \emptyset$, as otherwise $\{u_1, u_2\}$ is a $2$-cut.
By Lemma \ref{lem:2nd-neighborhood-independent}, we then have that $u_2 \sim u_1$. But now, we obtain a $3$-coloring of $G$ by coloring $N(v_2) \cup \{u_1\}$ with $1$, $N(u_1) \cup \{v_2, v_5\}$ with $2$, and $N(v_5) \setminus \{v_1\}$ with $3$, a contradiction.

Therefore, let $u_5 \in N^2(C)\cap N^2(v_5)$ and $c_5\in N(u_5)\cap N(v_5)$. If $c_5 = c_2$ then the Structure Algorithm applied to $C \cup \{a_1, b_1, u_1, c_2, u_2\}$ obtains a contradiction, so $c_5 \neq c_2$. Consider the cycle $C' = v_1b_1v_3v_4a_1v_1$. By the choice of $C$ there must exist $x \in N(\{v_2, v_5\})$ such that $x \nsim C'$. Up to symmetry, we may assume $x \sim v_5$. Then by Lemma \ref{lem:no-partial-diamond}, $x \sim N^2(C)$; so we may further assume $x = c_5$. Now, the Structure Algorithm applied to $C \cup \{a_1, b_1, u_1, c_2, c_5, u_2\}$ obtains a contradiction.
\end{proof}
 
Consider the 5-cycle $C' = v_1b_1v_3v_4v_5v_1$. By the choice of $C$, there must exist $a_2 \in N(v_2) \setminus N(C')$. If $d(v_2)=3$, then the Mixed Algorithm applied to $C \cup \{a_1, b_1, u_1\}$ obtains a contradiction, so $d(v_2)\geq 4$ and there exists another vertex $b_2\in N(v_2)$. Furthermore, $b_2 \sim \{v_4, v_5\}$, or else the Mixed Algorithm applied to $C \cup \{a_1, b_1, u_1, a_2, b_2\}$ obtains a contradiction. Let $i, j\in \{4, 5\}$ such that $b_2 \sim v_i$ and  $b_2 \nsim v_j$. Let $a_i \in N(v_i)$ and $a_j \in N(v_j)$.

We now show that $N^2(v_2) \cap N^2(C) \ne \emptyset$.  For, suppose $N^2(v_2) \cap N^2(C) = \emptyset$. Then, by Lemma \ref{lem:no-partial-diamond}, $N^2(v_i) \cap N^2(C) = \emptyset$ and hence $N^2(C)\subseteq N^2(v_j)\cup \{u_1\}$.
 Thus $N^3(C)=\emptyset$ and $|N^2(v_j)\cap N^2(C)|\leq 1$,
otherwise by Lemma \ref{lem:no-partial-diamond}, either $\{a_j,u_1\}$ or $(N^2(v_j)\cap N^2(C))\cup \{u_1\}$ is a cut set of size at most $2$.
By Lemma \ref{lem:2nd-neighborhood-independent}, $N^2(v_j)\cap N^2(C) =\{u_j\}$ for some $u_j\sim u_1$.
Note that similar to $b_2$, every vertex in $(N(v_2)\backslash V(C))\backslash \{a_2\}$ is adjacent to $v_i$. Furthermore, for every $a_i'\in N(v_i)\backslash V(C)$, we have $a_2\not\sim a_i'$, since otherwise either $(v_2, v_1 v_5, a_2 a_4', v_3, b_2)$ (when $i=4$) or
$(v_2, v_3 v_4, a_2 a_5', v_1, b_2)$ (when $i=5$) is a fork in $G$, giving a contradiction. Hence $N(v_2)\backslash V(C)\not\sim N(v_i)\backslash V(C)$.
Now, we obtain a $3$-coloring of $G$ by coloring $N(v_2) \cup (N(v_i)\setminus V(C)) \cup \{u_1\}$ with $1$, $((N(v_1) \cup N(v_3))\setminus V(C)) \cup \{v_2, v_j, u_j\}$ with $2$, and $(N(v_j)\setminus V(C)) \cup \{v_i\}$ with $3$, a contradiction.

Let $u_2 \in N^2(v_2) \cap N^2(C)$. Suppose there exists some vertex $u_1'\notin N(C)\cup \{u_2\}$ with $u_1'\sim u_1$. Then, by the Mixed Algorithm applied to $C \cup \{a_1, b_1, u_1, a_2, b_2, u_2, u_1'\}$, there exists some new vertex $w$ such that $u_1' \sim w$. Furthermore, if $w \in N^2(C)$, then the Structure Algorithm further applied to $C \cup \{a_1, b_1, u_1, a_2, b_2, u_2, u_1',w\}$ will return a contradiction. Therefore $w \in N(C)$, and hence $u_1' \in N^2(C)$. This implies that $u_1 \nsim N^3(C)$ (since every neighbor of $u_1$ in $N^{\geq 2}(C)$ is in $N^2(C)$ by the above argument). On the other hand, if there doesn't exist any $u_1'\notin N(C)\cup \{u_2\}$ with $u_1'\sim u_1$, then we also have $u_1\not\sim N^3(C)$.

Now note that $|N^2(C) \cap N^2(v_j) \setminus \{u_1, u_2\}| \leq 1$, otherwise the Mixed Algorithm applied to $C \cup \left(N^2(v_j) \cap N^2(C)\right) \cup \{a_1, b_1, u_1, a_2, b_2, u_2, a_j\}$ gives a contradiction (in either case for $j=4$ or $j=5$). It follows that $N^3(C)=\emptyset$, or else $(N^2(C) \cap N^2(v_j) \setminus \{u_1, u_2\}) \cup \{u_2\}$ is a cut set of size at most $2$. 

Now, if $N(v_j) \sim \{u_1, u_2\}$, $u_1 \sim u_2$ by Lemma \ref{lem:2nd-neighborhood-independent}, so we obtain a 3-coloring of $G$ by coloring $N(u_1) \cup \{v_2, v_i\}$ with 1, $N(v_2) \cup \{v_1, v_3\}$ with 2, and $\{v_j\}$ with 3; otherwise, we obtain a 3-coloring of $G$ by coloring $N(u_1) \cup \{v_2, v_i\} \cup (N^2(v_j) \cap N^2(C))$ with 1, $(N(u_2)\setminus N(u_1)) \cup \{u_1, v_j\}$ with 2, and $(N(v_j) \cap N(C)) \cup \{v_1, v_3, u_2\}$ with 3, a contradiction. This completes the proof of Lemma \ref{lem:diamond-2-attachment}.
\end{proof}

\begin{lemma}\label{lem:Bi_not_sim_2nd_nbhd}
Let $C$ be a good 5-cycle and let  
$u,v\in V(C)$ with $uv\notin E(C)$. Then $N(u)\cap N(v)\nsim N^2(C)$.
\end{lemma}

\begin{proof} Let $C=v_1v_2v_3v_4v_5v_1$. 
Suppose the assertion is false and assume without loss of generality that $a_1 \in N(v_1) \cap N(v_3)$ such that $a_1 \sim u_1$ for some $u_1 \in N^2(C)$. By Lemmas \ref{lem:no-partial-diamond} and \ref{lem:diamond-2-attachment}, $d(v_1) = d(v_3) =3$. Let $\sigma$ be a $3$-coloring $\sigma$ of $G \setminus \{v_3\}$, which exists by the minimality of $G$. Then $|\sigma(N(v_3))|=3$  or else we would be able to extend $\sigma$ to a $3$-coloring of $G$. Without loss of generality, let $\sigma(a_1) = 1$,  $\sigma(v_2) = 2$, and $\sigma(v_4) = 3$. Then  $\sigma(v_1) = 3$. 

We claim that $C':=v_1a_1v_3v_4v_5v_1$ is also a good 5-cycle. For, otherwise, there exist distinct $b,c\in N(v_2)\setminus N(C')$. By Lemma \ref{lem:no_neighborhood_containment} and the symmetry between $b$ and $c$, we may let $b'\in N(b)\setminus (N(c)\cup \{a_1\})$. Then $b'\nsim v_5$ or $b'\nsim v_4$ to avoid the triangle $b'v_4v_5$. Now $(v_2,bb',v_1v_5,v_3,c)$ (when $b'\nsim v_5$) or $(v_2,bb',v_3v_4,v_1,c)$ (when $b'\nsim v_4$) is a fork in $G$, a contradiction. 

Also note that by the choice of $C$, there
 exists a vertex $v_2' \in N(v_2) \setminus N(C')$. Hence, there is a symmetry between $a_1$ and $v_2$. Thus, we may assume without loss of generality that $\sigma(v_5)=1$ (if $\sigma(v_5) = 2$, then we can swap $v_2$ with $a_1$, swap $C$ with $C'$, and swap the colors $1$ and $2$). 
Then there exist $b_2 \in N(v_2) \cap \sigma^{-1}(1)$ and $b_2' \in N(b_2) \cap \sigma^{-1}(3)$. For, otherwise, we can extend $\sigma$ to a $3$-coloring of $G$ by recoloring $N(v_2)$ with $3$ (if necessary), recoloring $v_2$ with $1$, and coloring $v_3$ with $2$. 
     
  Similarly, there exists  $c_2 \in (N(v_2)\setminus V(C)) \cap \sigma^{-1}(3)$ and $c_2' \in N(c_2) \cap \sigma^{-1}(1)$. Otherwise, we can extend $\sigma$ to a $3$-coloring of $G$ by recoloring $v_1$ with $2$, recoloring $N(v_2)\setminus V(C)$ with $1$ (if necessary), recoloring $v_2$ with $3$, and coloring $v_3$ with $2$. 
     
 Observe that $b_2 \sim v_4$; for otherwise, $b_2'\ne v_4$ and $(v_2, b_2 b_2', v_3 v_4, v_1, c_2)$ would be a fork in $G$. Also, $c_2 \sim v_5$, or else $c_2'\ne v_5$ and $(v_2,v_1v_5,c_2c_2',b_2,v_3)$ would be a fork in $G$. Now by Lemmas
 \ref{lem:no-partial-diamond} and \ref{lem:diamond-2-attachment}, $N^2(C) \cap N^2(\{v_2, v_4, v_5\}) = \emptyset$. It follows that $N^2(C) \subseteq N(a_1)$. Thus, $\{a_1\}$ is a $1$-cut in $G$, contradicting Lemma \ref{lem:no-2-cut}.
\end{proof}

\setcounter{claim}{0}

\begin{lemma}\label{lem:one_neighbor_in_C}
Let $C$ be a good $5$-cycle in $G$ and let $u,v\in V(C)$ with $u\nsim v$. Then $|N(u)\cap N(v)| = 1$, i.e., there exists no vertex in $N(C)$ which is adjacent to both $u$ and $v$.
\end{lemma}

\begin{proof}
Let $C= v_1 v_2 v_3 v_4 v_5 v_1$ and for convenience $v_{j}\equiv v_{j\textrm{ mod }5}$ for every $j\in \mathbb{Z}$. Suppose for a contradiction and without loss of generality that there exists  $a_1 \in N(v_1) \cap N(v_{3})\setminus \{v_2\}$. By Lemma \ref{lem:Bi_not_sim_2nd_nbhd}, $a_1\nsim N^2(C)$.

\begin{claim}\label{cl:no_other_2_neighbors}
$N(v_i) \cap N(v_{i+2}) = \{v_{i+1}\}$ for all $i \neq 1$.
\end{claim}

\begin{proof}

Suppose there exists $i\ne 1$ and $b \in (N(v_i) \cap N(v_{i+2}))\backslash \{v_{i+1}\}$. Then $i \notin \{2, 5\}$. Otherwise, suppose by symmetry that $i=2$. Then $N^2(C) \subseteq N^2(v_5)$ by Lemma \ref{lem:Bi_not_sim_2nd_nbhd}. Hence, $N^3(C) = \emptyset$ or else $N(v_5)\setminus V(C)$ is a $1$-cut (if $d(v_5) = 3$) or $N^2(v_5) \cap N^2(C)$ is a $1$-cut (if $d(v_5) >3$), contradicting Lemma \ref{lem:no-2-cut}. But now, $N^2(C)$ must be an independent set, contradicting Lemma \ref{lem:2nd-neighborhood-independent}. 

Thus $i\in \{3,4\}$ and, by symmetry, we may assume $i=4$. Now, $N^2(C) \subseteq N^2(v_2) \cup N^2(v_5)$ by Lemma \ref{lem:Bi_not_sim_2nd_nbhd}. By Lemma \ref{lem:2nd-neighborhood-independent}, there must exist $u_2 \in N^2(C) \cap N^2(v_2)$ and $u_5 \in N^2(C) \cap N^2(v_5)$ such that $u_2 \sim u_5$. By Lemmas \ref{lem:diamond-2-attachment} and \ref{lem:no-2-converge-1},  $|N(v_2)\setminus V(C)| \leq 1$ or $|N^2(C) \cap N^2(v_2)| \leq 1$; so let $A=N(v_2)\setminus V(C)$ or $A=N^2(C)\cap N^2(v_2)$ such that $|A|\le 1$.  Similarly, let $B=N(v_5)\setminus V(C)$ or  $B=N^2(C) \cap N^2(v_5)$ such that $|B|\le 1$.  
Now, $N^3(C) = \emptyset$ and $N^2(C) = \{u_2, u_5\}$, or else $A \cup B$ is a $2$-cut in $G$, contradicting Lemma \ref{lem:no-2-cut}.

Let $a_2 \in N(v_2) \cap N(u_2)$ and $a_5 \in N(v_5) \cap N(u_5)$. Now $a_2 \nsim b$ or else the Mixed Algorithm applied to $C \cup \{b, a_1, a_2, u_2, a_5, u_5\}$ obtains a contradiction. By symmetry, $a_5 \nsim a_1$. But now, the Mixed Algorithm applied to $C \cup \{b, a_1, a_2, u_2, a_5, u_5\}$ obtains a contradiction.
\end{proof}

Note that $N(v_1) \nsim N(v_3)$; otherwise by Claim \ref{cl:no_other_2_neighbors}, for $b_1 \in N(v_1)$ and $b_3 \in N(v_3)$ with $b_1 \sim b_3$, $(v_1, b_1b_3, v_5v_4, v_2, a_1)$ would be a fork in $G$. 

\begin{claim}\label{cl:3-matching}
There exist three disjoint edges $u_ia_i$ for $i=2,4,5$ such that $a_i\in N(v_i)\setminus V(C)$ and $u_i\in N^2(C)$. 
\end{claim}
\begin{proof}
By Claim \ref{cl:no_other_2_neighbors} there exist distinct $a_i \in N(v_i)\setminus V(C)$ for $i \in \{2, 4, 5\}$. 
Moreover, note that by Lemma \ref{lem:no-partial-diamond} and Lemma \ref{lem:Bi_not_sim_2nd_nbhd}, $N^2(C) \subseteq N(\{a_2, a_4, a_5\})$.

Suppose there does not exist a matching of three edges from $\{a_2, a_4, a_5\}$ to $N^2(C)$. 
Then by Hall's theorem, either there exist distinct $i,j\in \{2,4,5\}$ such that $N^2(C) \cap N(a_i) = N^2(C)\cap N(a_j)=\{u\}$ for some $u\in N^2(C)$, or there exists some $i\in \{2,4,5\}$ such that $N^2(C) \cap N(a_i) = \emptyset$, or $|N(\{a_2, a_4, a_5\})\cap N^2(C)| = 2$ and $\{a_2, a_4, a_5\}$ is complete to $N^2(C)$. We claim that it suffices to consider the first two cases. Indeed, in the last case, $|N^2(C)|\leq 2$ and thus $|N^3(C)| = \emptyset$ (otherwise $N^2(C)$ is a cut of size at most $2$). By Lemma \ref{lem:2nd-neighborhood-independent}, $N^2(C)= \{u_1, u_2\}$ with $u_1\sim u_2$. This gives a contradiction since $G$ is triangle-free. Hence one of the two former cases above must happen.

In the former case,  let $k\in \{2, 4, 5\} \setminus\{i,j\}$. Now $N^3(C) = \emptyset$, or else $\{u, a_k\}$  (if $d(v_k) = 3$), or $\{u\} \cup (N^2(C) \cap N^2(v_k))$ (if $d(v_k) > 3$) is a cut set in $G$ of size at most 2, contradicting Lemma \ref{lem:no-2-cut}. By Lemma \ref{lem:2nd-neighborhood-independent}, $N^2(C)$ is not an independent set; so there exists $u_k \in N^2(C)$ such that $u_k \sim a_k$ and $u_k \sim u$. Let $u_k' \in N(u_k) \setminus  \{a_k, u\}$. Note that, by Lemma \ref{lem:no-2-converge-1}, $u_k'\in N(C)$  and, hence,  $u_k' \sim v_k$ and $N^2(C) \cap N^2(v_k) = \{u_k\}$. We obtain a $3$-coloring of $G$ by coloring $N(u)$ with $1$, coloring $N(u_k)$ with $2$, $\left(\left(N(v_1) \cup N(v_3)\right) \setminus V(C)\right) \cup \{v_2\}$ with $3$, and coloring $v_3$, $v_4$, $v_5$, and $v_1$ greedily in that order. This gives a contradiction. 

Thus, there exists $i\in \{2,4,5\}$ such that $N^2(C) \cap N(a_i) = \emptyset$. Let $\{j,k\} = \{2, 4, 5\} \setminus \{i\}$. Then by Lemma \ref{lem:no-2-converge-1}, $N^2(C) \subseteq N^2(v_j) \cup N^2(v_k)$. By Lemmas \ref{lem:diamond-2-attachment} and \ref{lem:no-2-converge-1},  $|N(v_j)\setminus V(C)| \leq 1$ or $|N^2(C) \cap N^2(v_j)| \leq 1$; so let $A=N(v_j)\setminus V(C)$ or $A=N^2(C)\cap N^2(v_j)$ such that $|A|\le 1$.  Similarly, let $B=N(v_k)\setminus V(C)$ or  $B=N^2(C) \cap N^2(v_k)$ such that $|B|\le 1$.  
Now, $N^3(C) = \emptyset$ and $|N^2(C) \cap N^2(v_j)| = |N^2(C) \cap N^2(v_k)| = 1$; otherwise, $N^3(C) = \emptyset$ and $N^2(C)$ is independent, contradicting Lemma \ref{lem:2nd-neighborhood-independent}, or $A \cup B$ is a $2$-cut in $G$, contradicting Lemma \ref{lem:no-2-cut}. Let $N^2(C) \cap N^2(v_j) = \{u_j\}$ and $N^2(C) \cap N^2(v_k) = \{u_k\}$, and let $a_j\in N(u_j)\cap N(v_j)$ and $a_k\in N(u_k)\cap N(v_k)$. Note that $u_j \sim u_k$ by Lemma \ref{lem:2nd-neighborhood-independent}.

By symmetry, we may suppose that $j = 5$. Let $\sigma$ be a $3$-coloring of $G \setminus (N(v_i)\setminus V(C))$, such that $N(u_5) \cup \{v_1\}$ is colored with $1$, $N(u_k) \cup \{v_3, v_5\}$ is colored with $2$, $(N(v_1)\setminus V(C)) \cup N(v_3)$ are colored with $3$.   

Then there must exist $c \in N(v_i)\setminus V(C)$ such that $c\sim N(v_j)\setminus V(C)$ and $c\sim N(v_k)\setminus V(C)$; as otherwise, we would be able to extend $\sigma$ to a $3$-coloring of $G$ by greedily coloring each vertex in $N(v_i)\setminus V(C)$ with $1$ or $2$. Without loss of generality, assume $c \sim a_5$ and $c \sim a_k$. Let $b_k \in N(u_k) \setminus \{a_k, u_5\}$; then $b_k\sim v_k$.  Now, $c \sim b_k$ to avoid the fork $(v_4, a_4c, v_5v_1, b_k, v_3)$ (if $i = 2$) or the fork $(v_2, a_2c, v_1v_5, b_k, v_3)$ (if $i=4$). But then $(c, a_5u_5, v_iv_3, a_k, b_k)$ is a fork in $G$, a contradiction. 
\end{proof}

Note that since $d(a_1) \geq 3$ and $a_1 \nsim N^2(C)$ by Lemma \ref{lem:Bi_not_sim_2nd_nbhd}, $a_1$ must have another neighbor in $N(C)$; without loss of generality we may suppose $a_1 \sim a_4$ or $a_1 \sim a_2$. If $a_1 \sim a_4$, the Mixed Algorithm applied to $C \cup \{a_1, a_2, u_2, a_4, u_4, a_5, u_5\}$ obtains a contradiction, so $a_1 \nsim a_4$ and by symmetry $a_1 \nsim a_5$, and hence $a_1 \sim a_2$. But now, the Mixed Algorithm applied to $C \cup \{a_1, a_2, u_2, a_4, u_4, a_5, u_5\}$ obtains a contradiction.
\end{proof}

\setcounter{claim}{0}

\section{Proof of Theorem \ref{thm:main_theorem}}\label{sec:fork_main}

\begin{proof}[Proof of Theorem \ref{thm:main_theorem}]
Let $C = v_1 v_2 v_3 v_4 v_5 v_1$ be a good $5$-cycle and for convenience $v_{j}\equiv v_{j\textrm{ mod }5}$ for every $j\in \mathbb{Z}$. By Lemma \ref{lem:one_neighbor_in_C}, every vertex outside $C$ has at most one neighbor on $C$. 

\begin{claim}\label{cl:matching_between_nbhds_degree_four}
For $i\in [5]$, $a \in N(v_{i+1})\setminus V(C)$, and distinct $a_i,b_i\in N(v_i)\setminus V(C)$, we have $a \sim \{a_i, b_i\}$. 
\end{claim}

\begin{proof}

For, otherwise, $(v_i, v_{i-1}v_{i-2}, v_{i+1} a, a_i, b_i)$ would be a fork in $G$, as by Lemma \ref{lem:one_neighbor_in_C}, $N(a)\cap V(C) = \{v_{i+1}\}$ and $N(a_i) \cap V(C) = N(b_i)\cap V(C) = \{v_i\}$.
\end{proof}

\begin{claim}\label{clm:distance_2_connection}
For any $i,j\in [5]$ with $j\not\in \{i-1, i, i+1\}$, if $d(v_i)\geq 4$ and $d(v_j)\geq 4$ then $N(v_i) \setminus V(C)\nsim N(v_j) \setminus V(C)$.
\end{claim}
\begin{proof}
Without loss of generality suppose $d(v_1)\geq 4$ and $d(v_3)\geq 4$. 
Note that, for any $i \in \{1,3\}$ and $x_i, y_i\in N(v_i)$, if there exists $z \in N(v_{4-i})$ such that $z \in N(x_i)\backslash N(y_i)$, then $(v_1, x_1 z, v_5v_4, v_2, y_1)$ (when $i=1$) or $(v_3, x_3 z, v_4 v_5, v_2, y_3)$ (when $i=3$) is a fork in $G$, a contradiction. 
Thus, if $N(v_1)\setminus V(C)\sim N(v_3)\setminus V(C)$, then $N(v_1)\setminus V(C)$ is complete to $N(v_3)\setminus V(C)$. By Claim \ref{cl:matching_between_nbhds_degree_four}, let $a_i \in N(v_i)\backslash V(C)$ for $i\in [3]$ such that $a_2\sim a_1$ and $a_2\sim a_3$. Now $a_2 a_1 a_3 a_2$ is a triangle in $G$, a contradiction.
\end{proof}

\begin{claim}\label{clm:distance_2_connection_strong} Suppose no degree 3 vertices of $G$ are consecutive on $C$. 
For any $i,j\in [5]$ with $j\not\in \{i-1, i, i+1\}$, if $d(v_i)\geq 4$, then $N(v_i) \setminus V(C)\nsim N(v_j) \setminus V(C)$.
\end{claim}
\begin{proof}
Without loss of generality, suppose $d(v_1)\geq 4$ and $N(v_1)\backslash V(C) \sim N(v_3)\backslash V(C)$. If $d(v_3)\geq 4$, then we are done by Claim \ref{clm:distance_2_connection}. Hence $d(v_3) = 3$. Let $a_3\in N(v_3)\setminus V(C)$ and $a_1, b_1\in N(v_1)\setminus V(C)$. Note that $a_3$ is complete to $N(v_1)\setminus V(C)$; for, otherwise, for any $x_1, y_1 \in N(v_1) \setminus V(C)$ such that $a_3 \in N(x_1)\backslash N(y_1)$, $(v_1, x_1 a_3, v_5 v_4, v_2, y_1)$ is a fork in $G$, a contradiction. Since $d(v_2)\geq 4$, let $a_2, b_2\in N(v_2) \setminus V(C)$. By Claim \ref{cl:matching_between_nbhds_degree_four} and without loss of generality assume that $a_3\sim a_2$ and $a_2\sim a_1$. Then $a_1 a_2 a_3 a_1$ is a triangle, a contradiction.
\end{proof}

We divide the proof into cases based on the number of degree $3$ vertices of $G$ that are contained in $V(C)$. 

\medskip 
{\it Case} 1. Every vertex on $C$ has degree $3$ in $G$. 

For $i\in [5]$, let $a_i\in N(v_i)\backslash V(C)$.
By the choice of $G$, $G-C$ admits a $3$-coloring, say $\sigma$.
We claim that for any choice of $\sigma$, $|\sigma(N(C))|=1$. For, otherwise, we may assume without loss of generality that  
$\sigma(a_1) \neq \sigma(a_{2})$. Then, we can modify $\sigma$ to a $3$-coloring of $G$ by coloring $v_1$ with $\sigma(a_{2})$ and coloring $v_5$, $v_4$, $v_3$, and $v_2$ greedily in that order, 
a contradiction.   

Without loss of generality we may assume that $\sigma(a_i) = 1$ for all $i \in [5]$; so $N(C)$ is independent.
For $i,j\in [5]$ with $i\neq j$, let $G_{ij}:= (G-C) + a_i a_j$. 

We claim that $G_{ij}$ is fork-free for all $i,j\in [5]$. Indeed, let $F$ be a fork in $G_{ij}$; so $a_ia_j\in E(F)$. Observe that neither $a_i$ nor $a_j$ can be a leaf in $F$; otherwise, suppose $a_i$ is a leaf in $F$, then $(F - a_i) + a_jv_j$ is a fork in $G$, a contradiction. Hence without loss of generality we can write $F = (a_i, a_ja_j', bb', c, d)$. Note that $\{a_j', b, c, d\} \cap N_G(C) = \emptyset$ since $N_G(C)$ is independent in $G$. Moreover, by Lemma \ref{lem:one_neighbor_in_C}, $b'\notin N(v_{i+1})\cap N(v_{i-1})$.  Then $(a_i, v_iv_{i+1}, bb', c, d)$ (when $b'\nsim v_{i+1}$) or $(a_i, v_iv_{i-1}, bb', c, d)$ (when $b'\nsim v_{i-1}$) is a fork in $G$, a contradiction. 

We now claim that for all $i,j\in [5]$, $a_i a_j$ is contained in a triangle in $G_{ij}$. For, otherwise,  $G_{ij}$ is triangle-free for some $i,j\in [5]$. Then by the minimality of $G$, $G_{ij}$ (and hence $G-C$) admits a $3$-coloring $\tau$ such that $\left|\tau\left(N(C)\right)\right| \geq 2$, a contradiction. 

Hence let $u_{ij} \in N(a_i) \cap N(a_j)$ for all $i,j \in [5]$.
We claim that there exists $u \in \{u_{ij}: i,j\in [5]\} \subseteq N^2(C)$ such that 
$|N(u)\cap N(C)|\ge 3$. Otherwise, $|N(u_{ij}) \cap \{a_k: k\in [5]\}| = 2$ for all $i,j \in [5]$; so $(a_1, u_{12}a_2, u_{13}a_3, u_{14}, u_{15})$ is a fork in $ G$, a contradiction. 

Thus let $\{a_i,a_j, a_k\}\subseteq N(u)$, where $i,j,k\in [5]$ are distinct. By symmetry between $i$ and $j$, assume that $k\not\in \{i-1, i+1\}$. Note that $d(u) = 3$, as for any $w \in N(u) \backslash \{a_i, a_j, a_k\}$, $(u, a_iv_i, a_kv_k, a_j, w)$ would be a fork in $G$. 

Now, there must exist $b \in (N_{G-C}(a_i)\backslash\{u\})\cap \sigma^{-1}(2)$, or else we can recolor $u$ with $3$ (if necessary), recolor $a_i$ with $2$, and obtain a $3$-coloring $\tau$ of $G-C$ with $\left|\tau\left(N(C)\right)\right| \geq 2$, a contradiction. Similarly, there  exist $c \in (N_{G-C}(a_i) \backslash \{u\})\cap \sigma^{-1}(3)$ and $c' \in N_{G-C}(c)\cap \sigma^{-1}(2)$; or else we can recolor $N_{G-C}(a_i)$ with $2$ (if necessary), recolor $a_i$ with $3$, and obtain a $3$-coloring $\tau$ of $G-C$ with $\left|\tau\left(N(C)\right)\right| \geq 2$, a contradiction. Now $(a_i, cc', v_iv_{i+1}, b, u)$ is a fork in $G$, a contradiction. 

\medskip

{\it Case} 2. Two consecutive vertices on $C$ have degree $3$ in $G$. 

By Case 1, we may assume that some vertex on $C$ has degree at least $4$ in $G$. So without loss of generality, we may assume that $d(v_4) = d(v_5) = 3$ and $d(v_1) \geq 4$. Let $a_1,b_1\in N(v_1)\setminus V(C)$ be distinct, and let $a_i\in N(v_i)\setminus V(C)$ for $i\in \{4,5\}$. 

By Claim \ref{cl:matching_between_nbhds_degree_four}, $a_5 \sim \{a_1, b_1\}$ and by symmetry assume $a_5 \sim a_1$. By the choice of $G$, $G-v_5$ admits a $3$-coloring, say $\sigma$. Then $|\sigma(N(v_5))|=3$ or else $\sigma$ can be extended to a $3$-coloring of $G$. Thus, we may assume $\sigma(v_1) = 1$, $\sigma(a_5) = 2$, and $\sigma(v_4) = 3$. This also implies that $\sigma(a_1) = 3$. 

We claim that $(N(v_1) \backslash \{a_1\})\cap \sigma^{-1}(3)=\emptyset$ for every choice of $\sigma$ with $\sigma(v_1) = 1$, $\sigma(a_5) = 2$, and $\sigma(v_4) = 3$. For, let $a \in (N(v_1) \backslash \{a_1\})\cap \sigma^{-1}(3)$. Then, there exist $b \in N(v_1)\cap \sigma^{-1}(2)$ and $b' \in N(b)\cap \sigma^{-1}(3)$. Otherwise, we could modify $\sigma$ to a $3$-coloring of $G$ by recoloring $N(v_1)\backslash \{v_5\}$ with $3$ (if necessary), recoloring $v_1$ with $2$, and coloring $v_5$ with $1$. Now $(v_1, bb', v_5v_4, a, a_1)$ is a fork in $G$, a contradiction.
    
  So $\sigma(b_1)=\sigma(v_2) = 2$ and, thus, $\sigma(v_3) = 1$. But now, the Coloring Algorithm with input $C \cup \{a_1, a_2, a_3, a_4, a_5, b_1\}$ and $\sigma$ obtains a contradiction.

\medskip

{\it Case} 3. Exactly two vertices on $C$ have degree $3$ in $G$. 

By the previous two cases, we may assume that $d(v_2) =d(v_5) = 3$ and that $d(v_i) \ge 4$ for all $i\in \{1,3,4\}$; let $a_i,b_i\in N(v_i)\setminus V(C)$ be distinct for all $i\in \{1,3,4\}$. Moreover, let $a_2\in N(v_2)\setminus V(C)$ and $a_5\in N(v_5)\setminus V(C)$. By Claim \ref{cl:matching_between_nbhds_degree_four}, we may suppose up to symmetry that $a_2 \sim b_3$, $a_5 \sim b_4$, and $a_5 \sim b_1$. Now, $a_3 \sim b_4$ or $a_3 \nsim b_4$, but in either case the Mixed Algorithm applied to $C \cup \{a_1, a_2, a_3, a_4, a_5, b_1, b_3, b_4\}$ obtains a contradiction.

\medskip

{\it Case} 4. At most one vertex on $C$ has degree 3 in $G$.
Without loss of generality, let   $a_i,b_i\in N(v_i)\setminus V(C)$ be distinct for $2\leq i\leq 5$, and let $a_1\in N(v_1)\setminus V(C)$. 
By Claim \ref{cl:matching_between_nbhds_degree_four} and Lemma \ref{lem:no_neighborhood_containment}, we may choose $a_i, b_i$ such that $a_i \sim a_{i+1}$ and $b_i\sim b_{i+1}$ for $2\leq i\leq 4$.

By Claim \ref{cl:matching_between_nbhds_degree_four}, $a_1 \sim \{a_2, b_2\}$ and $a_1 \sim \{a_5, b_5\}$. Without loss of generality assume that $a_1 \sim a_2$. 
Let $U_i:=N^2(v_i)\cap N^2(C)$ for $i\in [5]$. By Lemmas \ref{lem:no-partial-diamond} and \ref{lem:no-2-converge-1}, $|U_i| \leq 1$ for $2\leq i \leq 5$. Since $G$ is triangle-free, $U_i\cap U_{i+1} = \emptyset$ for $i\in [5]$. By Claim \ref{clm:distance_2_connection_strong}, $N(v_i)\setminus V(C)\nsim N(v_j)\setminus V(C)$ for $i,j\in [5]$ and $j\not\in \{i-1, i, i+1\}$. Note that $|U_1|\leq 1$ as well; for, otherwise, letting $u_1, u_1'\in U_1$ we have that $(a_1, a_2 v_2, a_5 v_5, u_1, u_1')$ (when $a_1\sim a_5$) or $(a_1, a_2 v_2, b_5 v_5, u_1, u_1')$ (when $a_1\sim b_5$) is a fork in $G$, a contradiction. By Lemma \ref{lem:no-partial-diamond}, $U_i$ is complete to $N(v_i)\backslash V(C)$ for $i\in [5]$.

We claim that $N^3(C) = \emptyset$. For, suppose there exists  $w \in N^3(C) \cap N(u_i)$, where $u_i\in U_i$ for some $i\in [5]$. 
If $a_1\sim a_5$ then $u_i\sim a_{i+2}$ to avoid the fork $(a_i, a_{i+1} a_{i+2}, u_i w, v_i, a_{i-1})$ and $u_i\sim a_{i-2}$ to avoid the fork $(a_i, a_{i-1} a_{i-2}, u_i w, v_i, a_{i+1})$, forcing the triangle $u_i a_{i+2} a_{i-2} u_i$ in $G$, a contradiction. Hence $a_1\nsim a_5$ and $a_1\sim b_5$. Similarly, we also obtain that $u_i\sim a_{i-2}$ and $u_i\sim a_{i+2}$ (by replacing some of the $a_j$'s by $b_j$'s in the forks and applying Lemma \ref{lem:no-partial-diamond}), giving a contradiction. Hence $N^3(C) = \emptyset$.

We now define a coloring $\sigma$ of $G$ as follows: color $((N(v_2) \cup N(v_5)) \setminus V(C)) \cup U_1 \cup \{v_1\}$ with 1, $N(v_3) \cup U_2 \cup U_4$ with 2, and $N(v_4) \cup (N(v_1)\backslash V(C)) \cup U_3 \cup U_5$ with 3 (if $U_i\cap U_j \neq \emptyset$, use the color of $U_{\min\{i,j\}}$).

 For each $i\in [5]$, let $u_i$ be the unique vertex in $U_i$ if $U_i\neq \emptyset$. Recall $U_i\cap U_{i+1}= \emptyset$ for $i\in [5]$. 
 But now $\sigma$ is a proper $3$-coloring of $G$ unless, up to symmetry (since $a_1 \sim \{a_5, b_5\}$ by Claim \ref{cl:matching_between_nbhds_degree_four}), $U_3\sim U_5$ and $u_3, u_5$ are both colored $3$ (note that $u_3 \nsim a_1$, otherwise $u_3$ would have been colored $1$). Moreover, $u_5 \nsim a_2$, otherwise $u_5$ would have been colored $2$. But now $(a_3, a_2 a_1, u_3 u_5, v_3, a_4)$ is a fork in $G$, a contradiction.
\end{proof}

\end{document}